\documentclass[10pt]{article}
\usepackage[inner=2cm,outer=2cm]{geometry} 

\usepackage[usenames, dvipsnames]{color}

\usepackage{amsmath} 
\usepackage{amssymb}  
\usepackage{amsthm}
\usepackage{adjustbox}
\usepackage[font={footnotesize,it}]{caption}
\usepackage{dsfont}
\usepackage{cite}
\usepackage{color}

\newtheorem{thm}{\bf{Theorem}}[section]

\newtheorem{lemma}{\bf{Lemma}}[section]

\newtheorem{remark}{\bf{Remark}}[section]

\newcommand{\nn}{\nonumber}
\newcommand{\vS}{\mathbf{S}}
\newcommand{\vY}{\mathbf{Y}}
\newcommand{\vy}{\mathbf{y}}
\newcommand{\vN}{\mathbf{N}}
\newcommand{\vnu}{\boldsymbol{\nu}}
\newcommand{\va}{\mathbf{a}}
\newcommand{\vepsilon}{\boldsymbol{\epsilon}}

\providecommand{\keywords}[1]{\textbf{\textit{Index terms---}} #1}

\allowdisplaybreaks

\begin{document}
\sloppy
\title{Hypothesis Testing under Subjective Priors and Costs as a Signaling Game\footnote{This research was supported in part by the Natural Sciences and Engineering Research Council (NSERC) of Canada. Part of this work \cite{cdc2018HT} was presented at the 57th IEEE Conference on Decision and Control (CDC 2018).}} 

\author{Serkan~Sar{\i}ta\c{s}$^1$ \and Sinan~Gezici$^2$ \and Serdar~Y\"uksel$^3$}
\date{%
	$^1$Division of Network and Systems Engineering, KTH Royal Institute of Technology, SE-10044, Stockholm, Sweden. Email: saritas@kth.se.\\%
	$^2$Department of Electrical and Electronics Engineering, Bilkent University, 06800, Ankara, Turkey. Email: gezici@ee.bilkent.edu.tr.\\
	$^3$Department of Mathematics and Statistics, Queen's University, K7L 3N6, Kingston, Ontario, Canada.  Email: yuksel@mast.queensu.ca.\\[2ex]%
}
\maketitle

\begin{abstract}
Many communication, sensor network, and networked control problems involve agents (decision makers) which have either misaligned objective functions or subjective probabilistic models. In the context of such setups, we consider binary signaling problems in which the decision makers (the transmitter and the receiver) have subjective priors and/or misaligned objective functions. Depending on the commitment nature of the transmitter to his policies, we formulate the binary signaling problem as a Bayesian game under either Nash or Stackelberg equilibrium concepts and establish equilibrium solutions and their properties. We show that there can be informative or non-informative equilibria in the binary signaling game under the Stackelberg and Nash assumptions, and derive the conditions under which an informative equilibrium exists for the Stackelberg and Nash setups. For the corresponding team setup, however, an equilibrium typically always exists and is always informative.  Furthermore, we investigate the effects of small perturbations in priors and costs on equilibrium values around the team setup (with identical costs and priors), and show that the Stackelberg equilibrium behavior is not robust to small perturbations whereas the Nash equilibrium is.
\end{abstract}



\section{INTRODUCTION}

In many decentralized and networked control problems, decision makers have either misaligned criteria or have subjective priors, which necessitates solution concepts from game theory. For example, detecting attacks, anomalies, and malicious behavior with regard to security in networked control systems can be analyzed under a game theoretic perspective, see e.g., \cite{sandberg2015cyberphysical, teixeira2015secure, networkSecurity,cyberSecuritySmartGrid, dan2010stealth, varshneyByzantine, varshneyHypothesisGame,detectionGameAdversary, multiAntennaJamming, gupta2010optimal, gupta2012dynamic, basar1985complete}. 

In this paper, we consider signaling games that refer to a class of two-player games of incomplete information in which an informed decision maker (transmitter or encoder) transmits information to another decision maker (receiver or decoder) in the hypothesis testing context. In the following, we first provide the preliminaries and introduce the problems considered in the paper, and present the related literature briefly.

\subsection{Notation}

We denote random variables with capital letters, e.g., $Y$, whereas possible realizations are shown by lower-case letters, e.g., $y$. The absolute value of scalar $y$ is denoted by $|y|$. The vectors are denoted by bold-faced letters, e.g., $\vy$. For vector $\vy$, $\vy^T$ denotes the transpose and $\|\vy\|$ denotes the Euclidean ($L_2$) norm. $\mathds{1}_{\{D\}}$ represents the indicator function of an event $D$, $\oplus$ stands for the exclusive-or operator, $\mathcal{Q}$ denotes the standard $\mathcal{Q}$-function; i.e.,  $\mathcal{Q}(x)={1\over\sqrt{2\pi}}\int_{x}^{\infty}\exp\{-{t^2\over 2}\}{\rm{d}}t$, and the sign of $x$ is defined as 
\[ \text{sgn}(x)=\begin{cases} 
-1 & \text{if }x < 0 \\
0 & \text{if }x=0 \\
1 & \text{if }x>0 
\end{cases}\,.\] 

\subsection{Preliminaries}

Consider a binary hypothesis-testing problem:
\begin{align}
\begin{split}
\mathcal{H}_0 : Y = S_0 + N \;, \\
\mathcal{H}_1 : Y = S_1 + N \;,
\label{eq:simpleMeasurement}
\end{split}
\end{align}
where $Y$ is the observation (measurement) that belongs to the observation set $\Gamma=\mathbb{R}$, $S_0$ and $S_1$ denote the deterministic signals under hypothesis $\mathcal{H}_0$ and hypothesis $\mathcal{H}_1$, respectively, and $N$ represents Gaussian noise; i.e., $N \sim \mathcal{N} (0,\sigma^2)$. In the Bayesian setup, it is assumed that the prior probabilities of $\mathcal{H}_0$ and $\mathcal{H}_1$ are available, which are denoted by $\pi_0$ and $\pi_1$, respectively, with $\pi_0+\pi_1=1$.

In the conventional Bayesian framework, the aim of the receiver is to design the optimal decision rule (detector) based on $Y$ in order to minimize the Bayes risk, which is defined as \cite{Poor}
\begin{align}
	r(\delta) = \pi_0 R_0(\delta) + \pi_1 R_1(\delta) \;,
	\label{eq:riskEq}
\end{align}
where $\delta$ is the decision rule, and $R_i(\cdot)$ is the conditional risk of the decision rule when hypothesis $\mathcal{H}_i$ is true for $i\in\{0,1\}$. In general, a decision rule corresponds to a partition of the observation set $\Gamma$ into two subsets $\Gamma_0$ and $\Gamma_1$, and the decision becomes $\mathcal{H}_i$ if the observation $y$ belongs to $\Gamma_i$, where $i\in\{0,1\}$.

The conditional risks in \eqref{eq:riskEq} can be calculated as
\begin{align}
	R_i(\delta) = C_{0i}\mathsf{P}_{0i} + C_{1i} \mathsf{P}_{1i} \;,
	\label{eq:condRiskEq}
\end{align}
for $i\in\{0,1\}$, where $C_{ji}\geq0$ is the cost of deciding for
$\mathcal{H}_j$ when $\mathcal{H}_i$ is true, and $\mathsf{P}_{ji}=\mathsf{Pr}(y\in\Gamma_j|\mathcal{H}_i)$ represents the conditional probability of deciding for $\mathcal{H}_j$ given that $\mathcal{H}_i$ is true, where $i,j\in\{0,1\}$ \cite{Poor}.

It is well-known that the optimal decision rule $\delta$ which minimizes the Bayes risk is the following test, known as the likelihood ratio test (LRT):
\begin{align}
\delta : \Bigg\{ \pi_1 (C_{01}-C_{11}) p_1(y) \overset{\mathcal{H}_1}{\underset{\mathcal{H}_0}{\gtreqless}} \pi_0 (C_{10}-C_{00})p_0(y) \;,
\label{eq:optimalReceiverDecision}
\end{align}
where $p_i(y)$ represents the probability density function (PDF)
of $Y$ under $\mathcal{H}_i$ for $i\in\{0,1\}$ \cite{Poor}.

If the transmitter and the receiver have the same objective function specified by \eqref{eq:riskEq} and \eqref{eq:condRiskEq}, then the signals can be designed to minimize the Bayes risk corresponding to the decision rule in \eqref{eq:optimalReceiverDecision}. This leads to a conventional formulation which has been studied intensely in the literature \cite{Poor,KayDetection}. 

On the other hand, it may be the case that the transmitter and the receiver can have non-aligned Bayes risks. In particular, the transmitter and the receiver may have different objective functions or priors: Let $C^t_{ji}$ and $C^r_{ji}$ represent the costs from the perspective of the transmitter and the receiver, respectively, where $i,j\in\{0,1\}$. Also let $\pi_i^t$ and $\pi_i^r$ for $i\in\{0,1\}$ denote the priors from the perspective of the transmitter and the receiver, respectively, with $\pi_0^j+\pi_1^j=1$, where $j\in\{t,r\}$. Here, from transmitter's and receiver's perspectives, the priors are assumed to be mutually absolutely continuous with respect to each other; i.e., $\pi_i^t=0\Rightarrow\pi_i^r=0$ and $\pi_i^r=0\Rightarrow\pi_i^t=0$ for $i\in\{0,1\}$. This condition assures that the impossibility of any hypothesis holds for both the transmitter and the receiver simultaneously. The aim of the transmitter is to perform the optimal design of signals $\mathcal{S}=\{S_0,S_1\}$ to minimize his Bayes risk; whereas, the aim of the receiver is to determine the optimal decision rule $\delta$ over all possible decision rules $\Delta$ to minimize his Bayes risk.

The Bayes risks are defined as follows for the transmitter and the receiver:
\begin{align}
r^j(\mathcal{S},\delta) = \pi_0^j (C^j_{00} \mathsf{P}_{00} + C^j_{10} \mathsf{P}_{10}) + \pi_1^j (C^j_{01} \mathsf{P}_{01} + C^j_{11} \mathsf{P}_{11})\;,
\label{eq:gameBayesRisk}
\end{align}
for $j\in\{t,r\}$. Here, the transmitter performs the optimal signal design problem under the power constraint below:
\begin{align*}
	\mathbb{S}\triangleq\{\mathcal{S}=\{S_0,S_1\}:| S_0 | ^2 \leq P_0 \,,\; | S_1 | ^2 \leq P_1\} \;,
\end{align*}
where $P_0$ and $P_1$ denote the power limits \cite[p. 62]{Poor}.

Although the transmitter and the receiver act sequentially in the game as described above, how and when the decisions are made and the nature of the commitments to the announced policies significantly affect the analysis of the equilibrium structure. Here, two different types of equilibria are investigated: 
\begin{enumerate}
	\item Nash equilibrium: the transmitter and the receiver make simultaneous decisions.
	\item Stackelberg equilibrium : the transmitter and the receiver make sequential decisions where the transmitter is the leader and the receiver is the follower.
\end{enumerate}
In this paper, the terms \textit{Nash game} and the \textit{simultaneous-move game} will be used interchangeably, and similarly, the \textit{Stackelberg game} and the \textit{leader-follower game} will be used interchangeably.

In the simultaneous-move game, the transmitter and the receiver announce their policies at the same time, and a pair of policies $(\mathcal{S}^*, \delta^*)$ is said to be a {\bf Nash equilibrium} \cite{basols99} if
\begin{align}
	\begin{split}
		r^t(\mathcal{S}^*, \delta^*) &\leq r^t(\mathcal{S}, \delta^*) \quad \forall \,\mathcal{S} \in \mathbb{S}\;, \\
		r^r(\mathcal{S}^*, \delta^*) &\leq r^r(\mathcal{S}^*, \delta) \quad \forall \,\delta \in \Delta\;.
		\label{eq:nashEquilibrium}
	\end{split}
\end{align}
As noted from the definition in \eqref{eq:nashEquilibrium}, under the Nash equilibrium, each individual player chooses an optimal strategy given the strategies chosen by the other player.

However, in the leader-follower game, the leader (transmitter) commits to and announces his optimal policy before the follower (receiver) does, the follower observes what the leader is committed to before choosing and announcing his optimal policy, and a pair of policies $(\mathcal{S}^*, \delta^*_{\mathcal{S}^*})$ is said to be a {\bf Stackelberg equilibrium} \cite{basols99} if
\begin{align}
	\begin{split}
		&r^t(\mathcal{S}^*, \delta^*_{\mathcal{S}^*}) \leq r^t(\mathcal{S}, \delta^*_\mathcal{S}) \quad \forall \,\mathcal{S} \in \mathbb{S}\;, \\
		&\text{where } \delta^*_\mathcal{S} \text{ satisfies} \\
		&r^r(\mathcal{S}, \delta^*_\mathcal{S}) \leq r^r(\mathcal{S}, \delta _\mathcal{S}) \quad \forall \,\delta_\mathcal{S} \in \Delta  \,.
		\label{eq:stackelbergEquilibrium}
	\end{split}
\end{align}
As observed from the definition in \eqref{eq:stackelbergEquilibrium}, the receiver takes his optimal action $\delta^*_\mathcal{S}$ after observing the policy of the transmitter $\mathcal{S}$. Further, in the Stackelberg game (also often called Bayesian persuasion games in the economics literature, see \cite{dynamicGameArxiv} for a detailed review), the leader cannot backtrack on his commitment, but he has a leadership role since he can manipulate the follower by anticipating the actions of the follower.

If an equilibrium is achieved when $\mathcal{S}^*$ is non-informative (e.g., $S_0^*=S_1^*$) and $\delta^*$ uses only the priors (since the received message is useless), then we call such an equilibrium a {\it non-informative (babbling) equilibrium} \cite[Theorem 1]{SignalingGames}.

\subsection{Two Motivating Setups}

We present two different scenarios that fit into the binary signaling context discussed here and revisit these setups throughout the paper\footnote{Besides the setups discussed here (and the throughout the paper), the deception game can also be modeled as follows. In the deception game, the transmitter aims to fool the receiver by sending deceiving messages, and this goal can be realized by adjusting the transmitter costs as $C^t_{00}>C^t_{10}$ and $C^t_{11}>C^t_{01}$; i.e, the transmitter is penalized if the receiver correctly decodes the original hypothesis. Similar to the standard communication setups, the goal of the receiver is to truly identify the hypothesis; i.e., $C^r_{00}<C^r_{10}$ and $C^r_{11}<C^r_{01}$.}. 

\subsubsection{Subjective Priors}\label{sec:inconPriors}

In almost all practical applications, there is some mismatch between the true and an assumed probabilistic system/data model, which results in performance degradation. This performance loss due to the presence of mismatch has been studied extensively in various setups (see e.g.,\cite{mismatchedEstimation}, \cite{estimationRobustnessMismatch}, \cite{mismatchSurvey} and references therein). In this paper, we have a further salient aspect due to decentralization, where the transmitter and the receiver have a mismatch. We note that in decentralized decision making, there have been a number of studies on the presence of a mismatch in the priors of decision makers \cite{BasTAC85, TeneketzisVaraiya88, CastanonTeneketzis88}. In such setups, even when the objective functions to be optimized are identical, the presence of subjective priors alters the formulation from a team problem to a game problem (see \cite[Section 12.2.3]{YukselBasarBook} for a comprehensive literature review on subjective priors also from a statistical decision making perspective). 

With this motivation, we will consider a setup where the transmitter and the receiver have different priors on the hypotheses $\mathcal{H}_0$ and $\mathcal{H}_1$, and the costs of the transmitter and the receiver are identical. In particular, from transmitter's perspective, the priors are $\pi_0^t$ and $\pi_1^t$, whereas the priors are $\pi_0^r$ and $\pi_1^r$ from receiver's perspective, and $C_{ji}=C^t_{ji}=C^r_{ji}$ for $i,j\in\{0,1\}$. We will investigate equilibrium solutions for this setup throughout the paper.

\subsubsection{Biased Transmitter Cost\protect\footnote{Here, the \textit{cost} refers to the objective function (Bayes risk), not the cost of a particular decision, $C_{ji}$. Note that, throughout the manuscript, the \textit{cost} refers to $C_{ji}$ except when it is used in the phrase \textit{Biased Transmitter Cost}.}}\label{sec:biasedCost}

A further application will be for a setup where the transmitter and the receiver have misaligned objective functions. Consider a binary signaling game in which the transmitter encodes a random binary signal $x=i$ as $\mathcal{H}_i$ by choosing the corresponding signal level $S_i$ for $i\in\{0,1\}$, and the receiver decodes the received signal $y$ as $u=\delta(y)$. Let the priors from the perspectives of the transmitter and the receiver be the same; i.e., $\pi_i=\pi_i^t=\pi_i^r$ for $i\in\{0,1\}$, and the Bayes risks of the transmitter and the receiver be defined as $r^t(\mathcal{S},\delta)=\mathbb{E}[\mathds{1}_{\{1=(x\oplus u\oplus b)\}}]$ and $r^r(\mathcal{S},\delta)=\mathbb{E}[\mathds{1}_{\{1=(x\oplus u)\}}]$, respectively, where $b$ is a random variable with a Bernoulli distribution; i.e., $\alpha\triangleq\mathsf{Pr}(b=0)=1-\mathsf{Pr}(b=1)$, and $\alpha$ can be translated as the probability that the Bayes risks (objective functions) of the transmitter and the receiver are aligned. Then, the following relations can be observed:
\begin{align*}
	r^t(\mathcal{S},\delta)&=\mathbb{E}[\mathds{1}_{\{1=(x\oplus u\oplus b)\}}]=\alpha(\pi_0\mathsf{P}_{10}+\pi_1\mathsf{P}_{01})+(1-\alpha)(\pi_0\mathsf{P}_{00}+\pi_1\mathsf{P}_{11}) \\
	&\Rightarrow \quad C^t_{01}=C^t_{10}=\alpha \text{ and } C^t_{00}=C^t_{11}=1-\alpha \;, \\
	r^r(\mathcal{S},\delta)&=\mathbb{E}[\mathds{1}_{\{1=(x\oplus u)\}}]=\pi_0\mathsf{P}_{10}+\pi_1\mathsf{P}_{01} \\
	&\Rightarrow \quad C^r_{01}=C^r_{10}=1 \text{ and } C^r_{00}=C^r_{11}=0 \;.
\end{align*}  

Note that, in the formulation above, the misalignment between the Bayes risks of the transmitter and the receiver is due to the presence of the bias term $b$ (i.e., the discrepancy between the Bayes risks of the transmitter and the receiver) in the Bayes risk of the transmitter. This can be viewed as an analogous setup to what was studied in a seminal work due to Crawford and Sobel \cite{SignalingGames}, who obtained the striking result that such a bias term in the objective function of the transmitter may have a drastic effect on the equilibrium characteristics; in particular, under regularity conditions, all equilibrium policies under a Nash formulation involve information hiding; for some extensions under quadratic criteria please see \cite{tacWorkArxiv} and \cite{omerHierarchial}.


\subsection{Related Literature}

In game theory, Nash and Stackelberg equilibria are drastically different concepts. Both equilibrium concepts find applications depending on the assumptions on the leader, that is, the transmitter, in view of the commitment conditions. Stackelberg games are commonly used to model attacker-defender scenarios in security domains \cite{securityStackelberg}. In many frameworks, the defender (leader) acts first by committing to a strategy, and the attacker (follower) chooses how and where to attack after observing defender's choice. However, in some situations, security measures may not be observable for the attacker; therefore, a simultaneous-move game is preferred to model such situations; i.e., the Nash equilibrium analysis is needed \cite{NashvsStackelberg}. These two concepts may have equilibria that are quite distinct: As discussed in \cite{tacWorkArxiv,dynamicGameArxiv}, in the Nash equilibrium case, building on \cite{SignalingGames}, equilibrium properties possess different characteristics as compared to team problems; whereas for the Stackelberg case, the leader agent is restricted to be committed to his announced policy, which leads to similarities with team problem setups \cite{CedricWork,omerHierarchial,akyolITapproachGame}. However, in the context of binary signaling, we will see that the distinction is not as sharp as it is in the case of quadratic signaling games \cite{tacWorkArxiv,dynamicGameArxiv}. 

Standard binary hypothesis testing has been extensively studied over several decades under different setups \cite{Poor,KayDetection}, which can also be viewed as a decentralized control/team problem involving a transmitter and a receiver who wish to minimize a common objective function. However, there exist many scenarios in which the analysis falls within the scope of game theory; either because the goals of the decision makers are misaligned, or because the probabilistic model of the system is not common knowledge among the decision makers.   

A game theoretic perspective can be utilized for hypothesis testing problem for a variety of setups. For example, detecting attacks, anomalies, and malicious behavior in network security can be analyzed under the game theoretic perspective \cite{sandberg2015cyberphysical, teixeira2015secure, networkSecurity,cyberSecuritySmartGrid, dan2010stealth}. In this direction, the hypothesis testing and the game theory approaches can be utilized together to investigate attacker-defender type applications \cite{varshneyByzantine,varshneyHypothesisGame,detectionGameAdversary, gupta2010optimal, gupta2012dynamic, basar1985complete, multiAntennaJamming}, multimedia source identification problems \cite{sourceIdentification}, inspection games \cite{inspectionGames,Avenhaus1994,avenhausInspectLeadership}, and deception games\cite{htDeception}. In \cite{varshneyHypothesisGame}, a Nash equilibrium of a zero-sum game between Byzantine (compromised) nodes and the fusion center (FC) is investigated. The strategy of the FC is to set the local sensor thresholds that are utilized in the likelihood-ratio tests, whereas the strategy of Byzantines is to choose their flipping probability of the bit to be transmitted. In \cite{detectionGameAdversary}, a zero-sum game of a binary hypothesis testing problem is considered over finite alphabets. The attacker has control over the channel, and the randomized decision strategy is assumed for the defender. The dominant strategies in Neyman-Pearson and Bayesian setups are investigated  under the Nash assumption. The authors of \cite{Avenhaus1994,avenhausInspectLeadership} investigate both Nash and Stackelberg equilibria of a zero-sum inspection game where an inspector (environmental agency) verifies, with the help of randomly sampled measurements, whether the amount of pollutant released by the inspectee (management of an industrial plant) is higher than the permitted ones. The inspector chooses a false alarm probability $\alpha$, and determines his optimal strategy over the set of all statistical tests with false alarm probability $\alpha$ to minimize the non-detection probability. On the other side, the inspectee chooses the signal levels (violation strategies) to maximize the non-detection probability. \cite{multiAntennaJamming} considers a complete-information zero-sum game between a centralized detection network and a jammer equipped with multiple antennas and investigates pure strategy Nash equilibria for this game. The fusion center (FC) chooses the optimal threshold of a single-threshold rule in order to minimize his error probability based on the observations coming from multiple sensors, whereas the jammer disrupts the channel in order to maximize FC's error probability under instantaneous power constraints. However, unlike the setups described above, in this work, we assume an additive Gaussian noise channel, and in the game setup, a Bayesian hypothesis testing setup is considered in which the transmitter chooses signal levels to be transmitted and the receiver determines the optimal decision rule. Both players aim to minimize their individual Bayes risks, which leads to a nonzero-sum game.\cite{htDeception} investigates the perfect Bayesian Nash equilibrium (PBNE) solution of a cyber-deception game in which the strategically deceptive interaction between the deceivee (privately-informed player, sender) and the deceiver (uninformed player, receiver) are modeled by a signaling game framework. It is shown that the hypothesis testing game admits no separating (pure, fully informative) equilibria, there exist only pooling and partially-separating-pooling equilibria; i.e., non-informative equilibria. Note that, in \cite{htDeception}, the received message is designed by the deceiver (transmitter), whereas we assume a Gaussian channel between the players. Further, the belief of the receiver (deceivee) about the priors is affected by the design choices of the transmitter (deceiver), unlike this setup, in which constant beliefs are assumed.
	
Within the scope of the discussions above, the binary signaling problem investigated here can be motivated under different application contexts: subjective priors and the presence of a bias in the objective function of the transmitter compared to that of the receiver. In the former setup, players have a common goal but subjective prior information, which necessarily alters the setup from a team problem to a game problem. The latter one is the adaptation of the biased objective function of the transmitter in \cite{SignalingGames} to the binary signaling problem considered here. We discuss these further in the following.

\subsection{Contributions}

The main contributions of this paper can be summarized as follows: (i) A game theoretic formulation of the binary signaling problem is established under subjective priors and/or subjective costs. (ii) The corresponding Stackelberg and Nash equilibrium policies are obtained, and their properties (such as uniqueness and informativeness) are investigated. It is proved that an equilibrium is almost always informative for a team setup, whereas in the case of subjective priors and/or costs, it may cease to be informative. (iii) Furthermore, robustness of equilibrium solutions to small perturbations in the priors or costs are established. It is shown that, the game equilibrium behavior around the team setup is robust under the Nash assumption, whereas it is not robust under the Stackelberg assumption. (iv) For each of the results, applications to two motivating setups (involving subjective priors and the presence of a bias in the objective function of the transmitter) are presented.
 
In the conference version of this study \cite{cdc2018HT}, some of the results (in particular, the Nash and Stackelberg equilibrium solutions and their robustness properties) appear without proofs. Here we provide the full proofs of the main theorems and also include the continuity analysis of the equilibrium. Furthermore, the setup and analysis presented in \cite{cdc2018HT} are extended to the multi-dimensional case and partially to the case with an average power constraint.
 
 
The remainder of the paper is organized as follows. The team setup, the Stackelberg setup, and the Nash setup of the binary signaling game are investigated in Sections II, Section III, and Section IV, respectively. In Section V, the multi-dimensional setup is studied, and in Section VI, the setup under an average power constraint is investigated. The paper ends with Section VII, where some conclusions are drawn and directions for future research highlighted. 

\section{TEAM THEORETIC ANALYSIS: CLASSICAL SETUP with IDENTICAL COSTS and PRIORS}

Consider the team setup where the costs and the priors are assumed to be the same and available for both the transmitter and the receiver; i.e., $C_{ji}=C^t_{ji}=C^r_{ji}$ and $\pi_i=\pi_i^t=\pi_i^r$ for $i,j\in\{0,1\}$. Thus the common Bayes risk becomes $r^t(\mathcal{S},\delta)=r^r(\mathcal{S},\delta)=\pi_0 (C_{00} \mathsf{P}_{00} + C_{10} \mathsf{P}_{10}) + \pi_1 (C_{01} \mathsf{P}_{01} + C_{11} \mathsf{P}_{11})$. The arguments for the proof of the following result follow from the standard analysis in the detection and estimation literature \cite{Poor,KayDetection}. However, for completeness, and for the relevance of the analysis in the following sections, a proof is included.
\begin{thm}
	Let $\tau\triangleq{\pi_0 (C_{10}-C_{00}) \over \pi_1 (C_{01}-C_{11})}$. If $\tau\leq0$ or $\tau=\infty$, the team solution of the binary signaling setup is non-informative. Otherwise; i.e., if $0<\tau<\infty$, the team solution is always informative.
	\label{thm:teamCase}
\end{thm}
\begin{proof}
The players adjust $S_0$, $S_1$, and $\delta$ so that $r^t(\mathcal{S},\delta)=r^r(\mathcal{S},\delta)$ is minimized. The Bayes risk of the transmitter and the receiver in \eqref{eq:gameBayesRisk} can be written as follows\footnote{Note that we are still keeping the parameters of the transmitter and the receiver as distinct in order to be able to utilize the expressions for the game formulations.}:
\begin{align}
\begin{split}
r^j(\mathcal{S},\delta) &= \pi_0^j C^j_{00} + \pi_1^j C^j_{11} + \pi_0^j (C^j_{10}-C^j_{00})\mathsf{P}_{10} + \pi_1^j (C^j_{01}-C^j_{11})\mathsf{P}_{01} \;,
\label{eq:updatedBayesRisk}		
\end{split}
\end{align}
for $j\in\{t,r\}$.

Here, first the receiver chooses the optimal decision rule $\delta^*_{S_0,S_1}$ for any given signal levels $S_0$ and $S_1$, and then the transmitter chooses the optimal signal levels $S_0^*$ and $S_1^*$ depending on the optimal receiver policy $\delta^*_{S_0,S_1}$. 

Assuming non-zero priors $\pi_0^t, \pi_0^r, \pi_1^t$, and $\pi_1^r$, the different cases for the optimal receiver decision rule can be investigated by utilizing \eqref{eq:optimalReceiverDecision} as follows:
\begin{enumerate}
	\item If $C^r_{01}>C^r_{11}$,
	\begin{enumerate}
		\item if $C^r_{10}>C^r_{00}$, the LRT in \eqref{eq:optimalReceiverDecision} must be applied to determine the optimal decision.
		\item if $C^r_{10}\leq C^r_{00}$, the left-hand side (LHS) of the inequality in \eqref{eq:optimalReceiverDecision} is always greater than the right-hand side (RHS); thus, the receiver always chooses $\mathcal{H}_1$.
	\end{enumerate}
	\item If $C^r_{01}=C^r_{11}$,
	\begin{enumerate}
		\item if $C^r_{10}>C^r_{00}$, the LHS of the inequality in \eqref{eq:optimalReceiverDecision} is always less than the RHS; thus, the receiver always chooses $\mathcal{H}_0$.
		\item if $C^r_{10}=C^r_{00}$, the LHS and RHS of the inequality in \eqref{eq:optimalReceiverDecision} are equal; hence, the receiver is indifferent of deciding $\mathcal{H}_0$ or $\mathcal{H}_1$. 
		\item if $C^r_{10}<C^r_{00}$, the LHS of the inequality in \eqref{eq:optimalReceiverDecision} is always greater than the RHS; thus, the receiver always chooses $\mathcal{H}_1$.
	\end{enumerate}
	\item If $C^r_{01}<C^r_{11}$,
	\begin{enumerate}
		\item if $C^r_{10}\geq C^r_{00}$, the LHS of the inequality in \eqref{eq:optimalReceiverDecision} is always less than the RHS; thus, the receiver always chooses $\mathcal{H}_0$.
		\item if $C^r_{10}<C^r_{00}$, the LRT in \eqref{eq:optimalReceiverDecision} must be applied to determine the optimal decision.
	\end{enumerate}
\end{enumerate}

The analysis above is summarized in Table~\ref{table:receiverCases}:
\begin{table}[ht]
\centering
\caption{Optimal decision rule analysis for the receiver.}
\label{table:receiverCases}
\begin{adjustbox}{max width=\textwidth}
	\begin{tabular}{|c|c|c|c|}
		\hline
		& \boldmath$C^r_{10}>C^r_{00}$ & \boldmath$C^r_{10}=C^r_{00}$ & \boldmath$C^r_{10}<C^r_{00}$ \\ \hline
		\boldmath$C^r_{01}>C^r_{11}$	& LRT & always $\mathcal{H}_1$ & always $\mathcal{H}_1$ \\ \hline
		\boldmath$C^r_{01}=C^r_{11}$	& always $\mathcal{H}_0$ & indifferent ($\mathcal{H}_0$ or $\mathcal{H}_1$) & always $\mathcal{H}_1$ \\ \hline
		\boldmath$C^r_{01}<C^r_{11}$	& always $\mathcal{H}_0$ & always $\mathcal{H}_0$ & LRT \\ \hline
	\end{tabular}
\end{adjustbox}
\end{table}

As it can be observed from Table~\ref{table:receiverCases}, the LRT is needed only when $\tau\triangleq{\pi_0^r (C^r_{10}-C^r_{00}) \over \pi_1^r (C^r_{01}-C^r_{11})}$ takes a finite positive value; i.e., $0<\tau<\infty$. Otherwise; i.e., $\tau\leq0$ or $\tau=\infty$, since the receiver does not consider any message sent by the transmitter, the equilibrium is non-informative.

For $0<\tau<\infty$, let $\zeta\triangleq\text{sgn}(C^r_{01}-C^r_{11})$ $($notice that $\zeta=\text{sgn}(C^r_{01}-C^r_{11})=\text{sgn}(C^r_{10}-C^r_{00})$ and $\zeta\in\{-1,1\})$. Then, the optimal decision rule for the receiver in \eqref{eq:optimalReceiverDecision} becomes
\begin{align}
\delta : \Bigg\{ \zeta{ p_1(y)\over p_0(y)} \overset{\mathcal{H}_1}{\underset{\mathcal{H}_0}{\gtreqless}} \zeta{\pi_0^r (C^r_{10}-C^r_{00}) \over \pi_1^r (C^r_{01}-C^r_{11})} = \zeta\tau \;.
\label{eq:lrt}
\end{align}
Let the transmitter choose optimal signals $\mathcal{S}=\{S_0,S_1\}$. Then the measurements in \eqref{eq:simpleMeasurement} become $\mathcal{H}_i : Y \sim \mathcal{N} (S_i,\sigma^2)$ for $i\in\{0,1\}$,
as $N\sim\mathcal{N} (0,\sigma^2)$, and the optimal decision rule for the receiver is obtained by utilizing \eqref{eq:lrt} as
\begin{align}
\delta^*_{S_0,S_1} &: \Bigg\{ \zeta y(S_1-S_0) \overset{\mathcal{H}_1}{\underset{\mathcal{H}_0}{\gtreqless}} \zeta\left(\sigma^2\ln(\tau)+{S_1^2-S_0^2\over2}\right)\;.
\label{eq:receiverLRT}
\end{align}
Since $\zeta Y(S_1-S_0)$ is distributed as $\mathcal{N} \Big(\zeta(S_1-S_0)S_i,(S_1-S_0)^2\sigma^2\Big)$ under $\mathcal{H}_i$ for $i\in\{0,1\}$, the conditional probabilities can be written based on \eqref{eq:receiverLRT} as follows:
\begin{align}
\mathsf{P}_{10}&=\mathsf{Pr}(y\in\Gamma_1|\mathcal{H}_0)=\mathsf{Pr}(\delta(y)=1|\mathcal{H}_0)=1-\mathsf{Pr}(\delta(y)=0|\mathcal{H}_0)=1-\mathsf{P}_{00}\nn\\
&=\mathcal{Q}\left(\zeta\left({\sigma\ln(\tau)\over|S_1-S_0|}+{|S_1-S_0|\over2\sigma}\right)\right)\;,
\label{eq:condProbs}
\end{align}
and similarly, $\mathsf{P}_{01}$ can be derived as $\mathsf{P}_{01}=\mathcal{Q}\left(\zeta\left(-{\sigma\ln(\tau)\over|S_1-S_0|}+{|S_1-S_0|\over2\sigma}\right)\right)$.

By defining $d\triangleq{|S_1-S_0|\over\sigma}$, $\mathsf{P}_{10}=\mathcal{Q}\left(\zeta\left({\ln(\tau)\over d}+{d\over2}\right)\right)$ and $\mathsf{P}_{01}=\mathcal{Q}\left(\zeta\left(-{\ln(\tau)\over d}+{d\over2}\right)\right)$ can be obtained. Then, the optimum behavior of the transmitter can be found by analyzing the derivative of the Bayes risk of the transmitter in \eqref{eq:updatedBayesRisk} with respect to $d$:
\begin{align}
\begin{split}
{\mathrm{d}\,r^t(\mathcal{S},\delta)\over\mathrm{d}\,d}&= -{1\over\sqrt{2\pi}}\exp\left\{-{(\ln\tau)^2\over 2d^2}\right\}\exp\left\{-{d^2\over8}\right\}\\
&\quad\;\times\Bigg(\pi_0^t\zeta (C^t_{10}-C^t_{00})\tau^{-{1\over2}}\left(-{\ln\tau\over d^2}+{1\over2}\right)+\pi_1^t\zeta (C^t_{01}-C^t_{11})\tau^{1\over2}\left({\ln\tau\over d^2}+{1\over2}\right)\Bigg)\;.
\label{eq:transmitterDerivative}
\end{split}
\end{align}

In \eqref{eq:transmitterDerivative}, if we utilize $C_{ji}=C^t_{ji}=C^r_{ji}$, $\pi_i=\pi_i^t=\pi_i^r$ and $\tau={\pi_0 (C_{10}-C_{00}) \over \pi_1 (C_{01}-C_{11})}$, we obtain the following:
\begin{align*}
	{\mathrm{d}\,r^t(\mathcal{S},\delta)\over\mathrm{d}\,d}&=-{1\over\sqrt{2\pi}}\exp\left\{-{(\ln\tau)^2\over 2d^2}\right\}\exp\left\{-{d^2\over8}\right\}\sqrt{\pi_0\pi_1(C_{10}-C_{00})(C_{01}-C_{11})}<0 \;.
\end{align*}	
Thus, in order to minimize the Bayes risk, the transmitter always prefers the maximum $d$, i.e., $d^*={\sqrt{P_0}+\sqrt{P_1}\over\sigma}$, and the equilibrium is informative. 
\end{proof}

\begin{remark} \label{rem:teamUnique}
	\begin{itemize}
		\item[(i)] Note that there are two informative equilibrium points which satisfy $d^*={\sqrt{P_0}+\sqrt{P_1}\over\sigma}$: $(S_0^*,S_1^*)=\left(-\sqrt{P_0},\sqrt{P_1}\right)$ and $(S_0^*,S_1^*)=\left(\sqrt{P_0},-\sqrt{P_1}\right)$, and the decision rule of the receiver is chosen based on the rule in \eqref{eq:receiverLRT} accordingly. Actually, these equilibrium points are essentially unique; i.e., they result in the same Bayes risks for the transmitter and the receiver.
		\item[(ii)] In the non-informative equilibrium, the receiver chooses either $\mathcal{H}_0$ or $\mathcal{H}_1$ as depicted in Table~\ref{table:receiverCases}. Since the message sent by the transmitter has no effect on the equilibrium, there are infinitely many ways of signal selection, which implies infinitely many equilibrium points. However, all these points are essentially unique; i.e., they result in the same Bayes risks for the transmitter and the receiver. Actually, if the receiver always chooses $\mathcal{H}_i$, the Bayes risks of the players are $r^j(\mathcal{S},\delta)=\pi^j_0C^j_{i0}+\pi_1^jC_{i1}^j$ for $i\in\{0,1\}$ and $j\in\{t,r\}$.
	\end{itemize}
\end{remark}

\section{STACKELBERG GAME ANALYSIS}

Under the Stackelberg assumption, first the transmitter (the leader agent) announces and commits to a particular policy, and then the receiver (the follower agent) acts accordingly. In this direction, first the transmitter chooses optimal signals $\mathcal{S}=\{S_0,S_1\}$ to minimize his Bayes risk $r^t(\mathcal{S},\delta)$, then the receiver chooses an optimal decision rule $\delta$ accordingly to minimize his Bayes risk $r^r(\mathcal{S},\delta)$. Due to the sequential structure of the Stackelberg game, besides his own priors and costs, the transmitter also knows the priors and the costs of the receiver so that he can adjust his optimal policy accordingly. On the other hand, besides his own priors and costs, the receiver knows only the policy and the action (signals $\mathcal{S}=\{S_0,S_1\}$) of the transmitter as he announces during the game-play; i.e., the costs and priors of the transmitter are not available to the receiver. 

\subsection{Equilibrium Solutions}
Under the Stackelberg assumption, the equilibrium structure of the binary signaling game can be characterized as follows:

\begin{thm}
If $\tau\triangleq{\pi_0^r (C^r_{10}-C^r_{00}) \over \pi_1^r (C^r_{01}-C^r_{11})}\leq0$ or $\tau=\infty$, the Stackelberg equilibrium of the binary signaling game is non-informative. Otherwise; i.e., if $0<\tau<\infty$, let $d\triangleq{|S_1-S_0|\over\sigma}$, $d_{\max}\triangleq{\sqrt{P_0}+\sqrt{P_1}\over\sigma}$, $\zeta\triangleq\text{sgn}(C^r_{01}-C^r_{11})$, $k_0\triangleq\pi_0^t\zeta(C^t_{10}-C^t_{00})\tau^{-{1\over2}}$, and $k_1\triangleq\pi_1^t\zeta(C^t_{01}-C^t_{11})\tau^{1\over2}$. Then, the Stackelberg equilibrium structure can be characterized as in Table~\ref{table:stackelbergSummary}, where $d^*=0$ stands for a non-informative equilibrium, and a nonzero $d^*$ corresponds to an informative equilibrium.
\begin{table*}[ht]
	\centering
	\caption{Stackelberg equilibrium analysis for $0<\tau<\infty$.}
	\label{table:stackelbergSummary}
	\begin{adjustbox}{max width=\textwidth}
		\begin{tabular}{|c|c|c|c|}
			\hline
			& \boldmath$\ln\tau\;(k_0-k_1)<0$ & \boldmath$\ln\tau\;(k_0-k_1)\geq0$ \\ \hline
			\boldmath$k_0+k_1<0$	& $d^*=\min\Big\{d_{\max}, \sqrt{\Big|{2\ln\tau(k_0-k_1)\over(k_0+k_1)}\Big|}\Big\}$ & $d^*=0$, non-informative  \\ \hline
			\boldmath$k_0+k_1\geq0$	& $d^*=d_{\max}$ & {$\!\begin{aligned}
				d_{\max}^2<\Big|{2\ln\tau(k_0-k_1)\over(k_0+k_1)}\Big|&\Rightarrow d^*=0\text{, non-informative} \\
				d_{\max}^2\geq\Big|{2\ln\tau(k_0-k_1)\over(k_0+k_1)}\Big|&\Rightarrow \left({k_1\over k_0\tau}\right)^{\text{sgn}(\ln(\tau))}\mathcal{Q}\left({|\ln(\tau)|\over d_{\max}}-{d_{\max}\over2}\right)-\mathcal{Q}\left({|\ln(\tau)|\over d_{\max}}+{d_{\max}\over2}\right)\overset{d^*=d_{\max}}{\underset{d^*=0}{\gtreqless}}0 \end{aligned}$} \\ \hline
		\end{tabular}
	\end{adjustbox}
\end{table*}
\label{thm:stackelbergCases}
\end{thm}
Before proving Theorem~\ref{thm:stackelbergCases}, we make the following remark:
\begin{remark}
	As we observed in Theorem~\ref{thm:teamCase}, for a team setup, an equilibrium is almost always informative (practically, $0<\tau<\infty$), whereas in the case of subjective priors and/or costs, it may cease to be informative. 
\end{remark}

\begin{proof}
	By applying the same case analysis as in the proof of Theorem~\ref{thm:teamCase}, it can be deduced that the equilibrium is non-informative if $\tau\leq0$ or $\tau=\infty$ (see Table~\ref{table:receiverCases}). Thus, $0<\tau<\infty$ can be assumed. 
	Then, from \eqref{eq:transmitterDerivative}, $r^t(\mathcal{S},\delta)$ is a monotone decreasing (increasing) function of $d$ if $k_0\left(-{\ln\tau\over d^2}+{1\over2}\right)+k_1\left({\ln\tau\over d^2}+{1\over2}\right)$, or equivalently $d^2(k_0+k_1)-2\ln\tau\,(k_0-k_1)$ is positive (negative) $\forall d$, where $k_0$ and $k_1$ are as defined in the theorem statement. Therefore, one of the following cases is applicable:
	\begin{enumerate}
	\item if $\ln\tau\;(k_0-k_1)<0$ and $k_0+k_1\geq0$, then $d^2 (k_0+k_1)>2\ln\tau(k_0-k_1)$ is satisfied $\forall d$, which means that $r^t(\mathcal{S},\delta)$ is a monotone decreasing function of $d$. Therefore, the transmitter tries to maximize $d$; i.e., chooses the maximum of $|S_1-S_0|$ under the constraints $| S_0 | ^2 \leq P_0$ and $| S_1 | ^2 \leq P_1$, hence $d^*=\max {|S_1-S_0|\over\sigma}={\sqrt{P_0}+\sqrt{P_1}\over\sigma}=d_{\max}$, which entails an informative equilibrium.
	\item if $\ln\tau\;(k_0-k_1)<0$, $k_0+k_1<0$, and $d_{\max}^2<\Big|{2\ln\tau(k_0-k_1)\over(k_0+k_1)}\Big|$, then $r^t(\mathcal{S},\delta)$ is a monotone decreasing function of $d$. Therefore, the transmitter maximizes $d$ as in the previous case. 
	\item if $\ln\tau\;(k_0-k_1)<0$, $k_0+k_1<0$, and $d_{\max}^2\geq\Big|{2\ln\tau(k_0-k_1)\over(k_0+k_1)}\Big|$, since $d^2(k_0+k_1)-2\ln\tau\,(k_0-k_1)$ is initially positive then negative, $r^t(\mathcal{S},\delta)$ is first decreasing and then increasing with respect to $d$. Therefore, the transmitter chooses the optimal $d^*$ such that $(d^*)^2=\Big|{2\ln\tau(k_0-k_1)\over(k_0+k_1)}\Big|$ which results in a minimal Bayes risk $r^t(\mathcal{S},\delta)$ for the transmitter. This is depicted in Figure~\ref{fig:StackelbergCase3}.
	\item if $\ln\tau\;(k_0-k_1)\geq0$ and $k_0+k_1<0$, then $d^2 (k_0+k_1)<2\ln\tau(k_0-k_1)$ is satisfied $\forall d$, which means that $r^t(\mathcal{S},\delta)$ is a monotone increasing function of $d$. Therefore, the transmitter tries to minimize $d$; i.e., chooses $S_0=S_1$ so that $d^*=0$. In this case, the transmitter does not provide any information to the receiver and the decision rule of the receiver in \eqref{eq:lrt} becomes $\delta : \zeta\overset{\mathcal{H}_1}{\underset{\mathcal{H}_0}{\gtreqless}} \zeta\tau$; i.e., the receiver uses only the prior information, thus the equilibrium is non-informative. 
	\item if $\ln\tau\;(k_0-k_1)\geq0$, $k_0+k_1\geq0$, and $d_{\max}^2<\Big|{2\ln\tau(k_0-k_1)\over(k_0+k_1)}\Big|$, then $r^t(\mathcal{S},\delta)$ is a monotone increasing function of $d$. Therefore, the transmitter chooses $S_0=S_1$ so that $d^*=0$. Similar to the previous case, the equilibrium is non-informative.
	\item if $\ln\tau\;(k_0-k_1)\geq0$, $k_0+k_1\geq0$, and $d_{\max}^2\geq\Big|{2\ln\tau(k_0-k_1)\over(k_0+k_1)}\Big|$,  $r^t(\mathcal{S},\delta)$ is first an increasing then a decreasing function of $d$, which makes the transmitter choose either the minimum $d$ or the maximum $d$; i.e., he chooses the one that results in a lower Bayes risk $r^t(\mathcal{S},\delta)$ for the transmitter. If the minimum Bayes risk is achieved when $d^*=0$, then the equilibrium is non-informative; otherwise (i.e., when the minimum Bayes risk is achieved when $d^*=d_{\max}$), the equilibrium is an informative one. There are three possible cases:
	\begin{enumerate}
		\item \underline{$\zeta(1-\tau)>0$} : 
		\begin{enumerate}
		\item If $d^*=0$, since $\delta : \zeta\overset{\mathcal{H}_1}{\underset{\mathcal{H}_0}{\gtreqless}} \zeta\tau$, the receiver always chooses $\mathcal{H}_1$, thus $\mathsf{P}_{10}=\mathsf{P}_{11}=1$ and $\mathsf{P}_{00}=\mathsf{P}_{01}=0$. Then, from \eqref{eq:updatedBayesRisk}, $r^t(\mathcal{S},\delta)=\pi_0^t C^t_{00} + \pi_1^t C^t_{11} + \pi_0^t (C^t_{10}-C^t_{00})$.
		\item If $d^*=d_{\max}$, by utilizing \eqref{eq:updatedBayesRisk} and \eqref{eq:condProbs}, $r^t(\mathcal{S},\delta)=\pi_0^t C^t_{00} + \pi_1^t C^t_{11} + \pi_0^t (C^t_{10}-C^t_{00})\mathcal{Q}\left(\zeta\left({\ln(\tau)\over d_{\max}}+{d_{\max}\over2}\right)\right) + \pi_1^t (C^t_{01}-C^t_{11})\mathcal{Q}\left(\zeta\left(-{\ln(\tau)\over d_{\max}}+{d_{\max}\over2}\right)\right)$.
		\end{enumerate}
	Then the decision of the transmitter is determined by the following:
	\begin{align}
	&\pi_0^t (C^t_{10}-C^t_{00})\overset{d^*=d_{\max}}{\underset{d^*=0}{\gtreqless}}\nn\\
	&\qquad\qquad\pi_0^t (C^t_{10}-C^t_{00})\mathcal{Q}\left(\zeta\left({\ln(\tau)\over d_{\max}}+{d_{\max}\over2}\right)\right)+ \pi_1^t (C^t_{01}-C^t_{11})\mathcal{Q}\left(\zeta\left(-{\ln(\tau)\over d_{\max}}+{d_{\max}\over2}\right)\right) \nn\\
	&\pi_0^t (C^t_{10}-C^t_{00})\mathcal{Q}\left(\zeta\left(-{\ln(\tau)\over d_{\max}}-{d_{\max}\over2}\right)\right)\overset{d^*=d_{\max}}{\underset{d^*=0}{\gtreqless}}\pi_1^t (C^t_{01}-C^t_{11})\mathcal{Q}\left(\zeta\left(-{\ln(\tau)\over d_{\max}}+{d_{\max}\over2}\right)\right) \nn\\
	\begin{split}
	&\zeta k_0\tau\mathcal{Q}\left(\zeta\left(-{\ln(\tau)\over d_{\max}}-{d_{\max}\over2}\right)\right)\overset{d^*=d_{\max}}{\underset{d^*=0}{\gtreqless}} \zeta k_1\mathcal{Q}\left(\zeta\left(-{\ln(\tau)\over d_{\max}}+{d_{\max}\over2}\right)\right) \,.
\label{eq:stackelbergCase6_1}
	\end{split}
	\end{align}
	For \eqref{eq:stackelbergCase6_1}, there are two possible cases:
	\begin{enumerate}
		\item \underline{$\zeta=1$ and $0<\tau<1$} : Since $\ln\tau(k_0-k_1)\geq0\Rightarrow k_0-k_1\leq0 $ and $k_0+k_1\geq0$, $k_1\geq0$ always. Then, \eqref{eq:stackelbergCase6_1} becomes
	\begin{align*}
	&{k_0\tau\over k_1}\mathcal{Q}\left(-{\ln(\tau)\over d_{\max}}-{d_{\max}\over2}\right)-\mathcal{Q}\left(-{\ln(\tau)\over d_{\max}}+{d_{\max}\over2}\right)\overset{d^*=d_{\max}}{\underset{d^*=0}{\gtreqless}}0 \,.
	\end{align*}
	\item \underline{$\zeta=-1$ and $\tau>1$} : Since $\ln\tau(k_0-k_1)\geq0\Rightarrow k_0-k_1\geq0 $ and $k_0+k_1\geq0$, $k_0\geq0$ always. Then, \eqref{eq:stackelbergCase6_1} becomes
	\begin{align*}
	    &{k_1\over k_0\tau}\mathcal{Q}\left({\ln(\tau)\over d_{\max}}-{d_{\max}\over2}\right)-\mathcal{Q}\left({\ln(\tau)\over d_{\max}}+{d_{\max}\over2}\right)\overset{d^*=d_{\max}}{\underset{d^*=0}{\gtreqless}}0 \,.
	\end{align*}
	\end{enumerate}
    \item \underline{$\zeta(1-\tau)=0 \Leftrightarrow \tau=1$} :  Since $k_0+k_1\geq0$ and $d^2(k_0+k_1)-2\ln\tau\,(k_0-k_1)\geq0$, $r^t(\mathcal{S},\delta)$ is a monotone decreasing function of $d$, which implies $d^*=d_{\max}$ and informative equilibrium.
		\item \underline{$\zeta(1-\tau)<0$} : 
\begin{enumerate}
	\item If $d^*=0$, since $\delta : \zeta\overset{\mathcal{H}_1}{\underset{\mathcal{H}_0}{\gtreqless}} \zeta\tau$, the receiver always chooses $\mathcal{H}_0$, thus $\mathsf{P}_{00}=\mathsf{P}_{01}=1$ and $\mathsf{P}_{10}=\mathsf{P}_{11}=0$. Then, from \eqref{eq:updatedBayesRisk}, $r^t(\mathcal{S},\delta)=\pi_0^t C^t_{00} + \pi_1^t C^t_{11} + \pi_1^t (C^t_{01}-C^t_{11})$.
	\item If $d^*=d_{\max}$, by utilizing \eqref{eq:updatedBayesRisk} and \eqref{eq:condProbs}, $r^t(\mathcal{S},\delta)=\pi_0^t C^t_{00} + \pi_1^t C^t_{11} + \pi_0^t (C^t_{10}-C^t_{00})\mathcal{Q}\left(\zeta\left({\ln(\tau)\over d_{\max}}+{d_{\max}\over2}\right)\right) + \pi_1^t (C^t_{01}-C^t_{11})\mathcal{Q}\left(\zeta\left(-{\ln(\tau)\over d_{\max}}+{d_{\max}\over2}\right)\right)$.
\end{enumerate}
Then, similar to the analysis in case-a), the decision of the transmitter is determined by the following:
\begin{align}
\begin{split}
&\zeta k_1\mathcal{Q}\left(\zeta\left({\ln(\tau)\over d_{\max}}-{d_{\max}\over2}\right)\right)\overset{d^*=d_{\max}}{\underset{d^*=0}{\gtreqless}}\zeta k_0\tau\mathcal{Q}\left(\zeta\left({\ln(\tau)\over d_{\max}}+{d_{\max}\over2}\right)\right) \,.
\label{eq:stackelbergCase6_2}
\end{split}
\end{align}
For \eqref{eq:stackelbergCase6_2}, there are two possible cases:
\begin{enumerate}
	\item \underline{$\zeta=-1$ and $0<\tau<1$} : Since $\ln\tau(k_0-k_1)\geq0\Rightarrow k_0-k_1\leq0 $ and $k_0+k_1\geq0$, $k_1\geq0$ always. Then, \eqref{eq:stackelbergCase6_2} becomes
	\begin{align*}
	&{k_0\tau\over k_1}\mathcal{Q}\left(-{\ln(\tau)\over d_{\max}}-{d_{\max}\over2}\right)-\mathcal{Q}\left(-{\ln(\tau)\over d_{\max}}+{d_{\max}\over2}\right)\overset{d^*=d_{\max}}{\underset{d^*=0}{\gtreqless}}0 \,.
	\end{align*}
	\item \underline{$\zeta=1$ and $\tau>1$} : Since $\ln\tau(k_0-k_1)\geq0\Rightarrow k_0-k_1\geq0 $ and $k_0+k_1\geq0$, $k_0\geq0$ always. Then, \eqref{eq:stackelbergCase6_2} becomes
	\begin{align*}
	&{k_1\over k_0\tau}\mathcal{Q}\left({\ln(\tau)\over d_{\max}}-{d_{\max}\over2}\right)-\mathcal{Q}\left({\ln(\tau)\over d_{\max}}+{d_{\max}\over2}\right)\overset{d^*=d_{\max}}{\underset{d^*=0}{\gtreqless}}0 \,.
	\end{align*}
\end{enumerate}
	\end{enumerate} 
	Thus, by combining all the cases, the comparison of the transmitter Bayes risks for $d^*=0$ and $d^*=d_{\max}$ reduces to the following rule:
	\begin{align}
	\begin{split}
		&\left({k_1\over k_0\tau}\right)^{\text{sgn}(\ln(\tau))}\mathcal{Q}\left({|\ln(\tau)|\over d_{\max}}-{d_{\max}\over2}\right)-\mathcal{Q}\left({|\ln(\tau)|\over d_{\max}}+{d_{\max}\over2}\right)\overset{d^*=d_{\max}}{\underset{d^*=0}{\gtreqless}}0\,.
		\label{eq:stackelbergCase9}
	\end{split}
	\end{align}    
	\end{enumerate}
\end{proof}
The most interesting case is Case-3 in which $\ln\tau\;(k_0-k_1)<0, k_0+k_1<0,$ and $d_{\max}^2\geq\Big|{2\ln\tau(k_0-k_1)\over(k_0+k_1)}\Big|$, since in all other cases, the transmitter chooses either the minimum or the maximum distance between the signal levels. Further, for classical hypothesis-testing in the team setup, the optimal distance corresponds to the maximum separation \cite{Poor}. However, in Case-3, there is an optimal distance $d^*=\sqrt{\Big|{2\ln\tau(k_0-k_1)\over(k_0+k_1)}\Big|}<d_{\max}$ that makes the Bayes risk of the transmitter minimum as it can be seen in Figure~\ref{fig:StackelbergCase3}.
\begin{figure}[ht]
	\centering
	\includegraphics[width=0.5\linewidth]{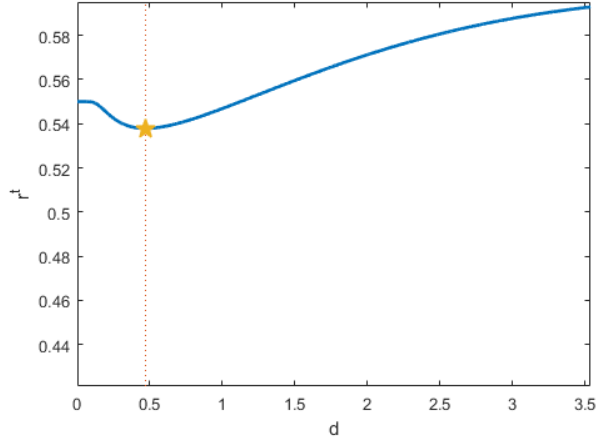}
	\caption{The Bayes risk of the transmitter versus $d$ when $ C^t_{00}=0.6, C^t_{10}=0.4, C^t_{01}=0.4, C^t_{11}=0.6, C^r_{00}=0, C^r_{10}=0.9, C^r_{01}=0.4, C^r_{11}=0, \pi_0^t=0.25, \pi_0^r=0.25, P_0=1, P_1=1$, and $\sigma = 0.1$. The optimal $d^*=\sqrt{\Big|{2\ln\tau(k_0-k_1)\over(k_0+k_1)}\Big|}=0.4704<d_{max}=20$ and its corresponding Bayes risk $r^t=0.5379$ are indicated by the star.}
	\label{fig:StackelbergCase3}
\end{figure}

\begin{remark}
Similar to the team setup analysis, for every possible case in Table~\ref{table:stackelbergSummary}, there are more than one equilibrium points, and they are essentially unique since the Bayes risks of the transmitter and the receiver depend on $d$. In particular,
\begin{itemize}
	\item[(i)] for $d^*=d_{\max}$, the equilibrium is informative, $(S_0^*,S_1^*)=\left(-\sqrt{P_0},\sqrt{P_1}\right)$ and $(S_0^*,S_1^*)=\left(\sqrt{P_0},-\sqrt{P_1}\right)$ are the only possible choices for the transmitter, which are essentially unique, and the decision rule of the receiver is chosen based on the rule in \eqref{eq:receiverLRT}.
	\item[(ii)] for $d^*=\sqrt{\Big|{2\ln\tau(k_0-k_1)\over(k_0+k_1)}\Big|}$, the equilibrium is informative, there are infinitely many choices for the transmitter and the receiver, and all of them are essentially unique; i.e., they result in the same Bayes risks for the transmitter and the receiver.
	\item[(iii)] for $d^*=0$ or $\tau\notin(0,\infty)$, the equilibrium is non-informative and there are infinitely many equilibrium points which are essentially unique; see Remark~\ref{rem:teamUnique}-(ii). 
\end{itemize}   
\end{remark}

\subsection{Continuity and Robustness to Perturbations around the Team Setup} \label{sec:ConvRobustStackelberg}
We now investigate the effects of small perturbations in priors and costs on equilibrium values. In particular, we consider the perturbations around the team setup; i.e., at the point of identical priors and costs. 

Define the perturbation around the team setup as $\vepsilon=\{\epsilon_{\pi0},\epsilon_{\pi1},\epsilon_{00},\epsilon_{01},\epsilon_{10},\epsilon_{11}\}\in\mathbb{R}^6$ such that $\pi_i^t=\pi_i^r+\epsilon_{\pi i}$ and $C_{ji}^t=C_{ji}^r+\epsilon_{ji}$ for $i,j\in\{0,1\}$ (note that the transmitter parameters are perturbed around the receiver parameters which are assumed to be fixed). Then, for $0<\tau<\infty$, at the point of identical priors and costs, small perturbations in both priors and costs imply $k_0=(\pi_0^r+\epsilon_{\pi 0})\zeta(C^r_{10}-C^r_{00}+\epsilon_{10}-\epsilon_{00})\tau^{-{1\over2}}$ and $k_1=(\pi_1^r+\epsilon_{\pi 1})\zeta(C^r_{01}-C^r_{11}+\epsilon_{01}-\epsilon_{11})\tau^{1\over2}$. Since, for $0<\tau<\infty$, $k_0=k_1=\sqrt{\pi^r_0\pi^r_1}\sqrt{(C^r_{10}-C^r_{00})(C^r_{01}-C^r_{11})}>0$ at the point of identical priors and costs, it is possible to obtain both positive and negative $(k_0-k_1)$ by choosing the appropriate perturbation $\vepsilon$ around the team setup. Then, as it can be observed from Table~\ref{table:stackelbergSummary}, even the equilibrium may alter from an informative one to a non-informative one; hence, under the Stackelberg equilibrium, the policies are not continuous with respect to small perturbations around the point of identical priors and costs, and the equilibrium behavior is not robust to small perturbations in both priors and costs. 

\subsection{Application to the Motivating Examples}

\subsubsection{Subjective Priors}
Referring to Section~\ref{sec:inconPriors}, for $0<\tau<\infty$, the related parameters can be found as follows (note that the equilibrium is non-informative if $\tau\leq0$ or $\tau=\infty$): 
	\begin{align*}
	\tau&={\pi_0^r (C_{10}-C_{00}) \over \pi_1^r (C_{01}-C_{11})} \,,\\
	k_0&=\pi_0^t\sqrt{\pi_1^r\over\pi_0^r}\sqrt{(C_{10}-C_{00})(C_{01}-C_{11})}\,,\\
	k_1&=\pi_1^t\sqrt{\pi_0^r\over\pi_1^r}\sqrt{(C_{10}-C_{00})(C_{01}-C_{11})}\,.
	\end{align*}
Since $k_0+k_1>0$, depending on the values of $\ln\tau\;(k_0-k_1)$, $d_{\max}^2$, and $\Big|{2\ln\tau(k_0-k_1)\over(k_0+k_1)}\Big|$, Case-1, Case-5 or Case-6 of Theorem~\ref{thm:stackelbergCases} may hold as depicted in Table~\ref{table:stackelbergInconsistentPriors}. Here, the decision rule in Case-6 is the same as \eqref{eq:stackelbergCase9}.
	\begin{table*}[ht]
		\centering
		\caption{Stackelberg equilibrium analysis of subjective priors case for $0<\tau<\infty$.}
		\label{table:stackelbergInconsistentPriors}
		\begin{adjustbox}{max width=\textwidth}
			\begin{tabular}{|c|c|c|c|}
				\hline
				& \boldmath$0<\tau<1$ & \boldmath$1\leq\tau<\infty$ \\ \hline
				\boldmath${\pi_0^t\over\pi_1^t}<{\pi_0^r\over\pi_1^r}$	& {$\!\begin{aligned}
					d_{\max}^2<\Big|{2\ln\tau(k_0-k_1)\over(k_0+k_1)}\Big|&\Rightarrow \text{ Case-5 applies, } d^*=0\text{, non-informative} \\
					d_{\max}^2\geq\Big|{2\ln\tau(k_0-k_1)\over(k_0+k_1)}\Big|&\Rightarrow\text{ Case-6 applies}   \end{aligned}$} & Case-1 applies, $d^*=d_{\max}$  \\ \hline
				\boldmath${\pi_0^t\over\pi_1^t}\geq{\pi_0^r\over\pi_1^r}$	& Case-1 applies, $d^*=d_{\max}$ & {$\!\begin{aligned}
					d_{\max}^2<\Big|{2\ln\tau(k_0-k_1)\over(k_0+k_1)}\Big|&\Rightarrow \text{ Case-5 applies, } d^*=0\text{, non-informative} \\
					d_{\max}^2\geq\Big|{2\ln\tau(k_0-k_1)\over(k_0+k_1)}\Big|&\Rightarrow\text{ Case-6 applies}   \end{aligned}$} \\ \hline
			\end{tabular}
		\end{adjustbox}
	\end{table*}
	
\subsubsection{Biased Transmitter Cost}
Based on the arguments in Section~\ref{sec:biasedCost}, the related parameters can be found as follows:
	\begin{align*}
	\tau={\pi_0\over\pi_1} \,,\; k_0=\sqrt{\pi_0\pi_1}(2\alpha-1)\,,\;
	k_1=\sqrt{\pi_0\pi_1}(2\alpha-1)\,.
	\end{align*}
	Then, $\ln\tau\;(k_0-k_1)=0$ and $k_0+k_1=2\sqrt{\pi_0\pi_1}(2\alpha-1)$; hence, either Case-4 or Case-6 of Theorem~\ref{thm:stackelbergCases} applies. Namely, if $\alpha<1/2$ (Case-4 of Theorem~\ref{thm:stackelbergCases} applies), the transmitter chooses $S_0=S_1$ to minimize $d$ and the equilibrium is non-informative; i.e., he does not send any meaningful information to the transmitter and the receiver considers only the priors. If $\alpha=1/2$, the transmitter has no control on his Bayes risk, hence the  equilibrium is non-informative. Otherwise; i.e., if $\alpha>1/2$ (Case-6 of Theorem~\ref{thm:stackelbergCases} applies), the equilibrium is always informative. In other words, if $\alpha>1/2$, the players act like a team. As it can be seen, the informativeness of the equilibrium depends on $\alpha=\mathsf{Pr}(b=0)$, the probability that the Bayes risks of the transmitter and the receiver are aligned. 

\section{NASH GAME ANALYSIS}

Under the Nash assumption, the transmitter chooses optimal signals $\mathcal{S}=\{S_0,S_1\}$ to minimize $r^t(\mathcal{S},\delta)$, and the receiver chooses optimal decision rule $\delta$ to minimize $r^r(\mathcal{S},\delta)$ simultaneously. In this Nash setup, the transmitter and the receiver do not need to know the priors and the costs of each other; they need to know only their own priors and costs while calculating the best response to a given action of other player. Further, there is no commitment between the transmitter and the receiver. Due to this difference, the equilibrium structure and robustness properties of the Nash equilibrium show significant differences from the ones in the Stackelberg equilibrium, as stated in the following.

In the analysis, we assume deterministic policies for the transmitter and receiver, and we restrict the receiver to use only the single-threshold rules. Although a single-threshold rule is sub-optimal for the receiver in general, it is always optimal for Gaussian densities, and always optimal for uni-modal densities under the maximum likelihood decision rule \cite{Poor, MuratA}.

\subsection{Equilibrium Solutions}
Under the Nash assumption, the equilibrium structure of the binary signaling game can be characterized as follows:
\begin{thm}
	Let $\tau\triangleq{\pi_0^r (C^r_{10}-C^r_{00}) \over \pi_1^r (C^r_{01}-C^r_{11})}$ and $\zeta\triangleq\text{sgn}(C^r_{01}-C^r_{11})$, $\xi_0 \triangleq {C^t_{10}-C^t_{00}\over C^r_{10}-C^r_{00}}$, and $\xi_1 \triangleq {C^t_{01}-C^t_{11}\over C^r_{01}-C^r_{11}}$. If $\tau\leq0$ or $\tau=\infty$, then the Nash equilibrium of the binary signaling game is non-informative. Otherwise; i.e., if $0<\tau<\infty$, the Nash equilibrium structure is as depicted in Table~\ref{table:nashSummary}.
	\begin{table*}[ht]
		\centering
		\caption{Nash equilibrium analysis for $0<\tau<\infty$.}
		\label{table:nashSummary}
		\begin{adjustbox}{max width=\textwidth}
			\begin{tabular}{|c|c|c|c|}
				\hline
				& \boldmath$\xi_0>0$ & \boldmath$\xi_0=0$ & \boldmath$\xi_0<0$ \\ \hline
				\boldmath$\xi_1>0$	& unique informative equilibrium & non-informative equilibrium & {$\!\begin{aligned}
					P_0>P_1&\Rightarrow \text{ non-informative equilibrium} \\
					P_0=P_1&\Rightarrow \text{ non-informative equilibrium} \\
					P_0<P_1&\Rightarrow \text{ unique informative equilibrium}
					\end{aligned}$} \\ \hline
				\boldmath$\xi_1=0$	& non-informative equilibrium & non-informative equilibrium & non-informative equilibrium \\ \hline
				\boldmath$\xi_1<0$	& {$\!\begin{aligned}
					P_0>P_1&\Rightarrow \text{ unique informative equilibrium} \\
					P_0=P_1&\Rightarrow \text{ non-informative equilibrium} \\
					P_0<P_1&\Rightarrow \text{ non-informative equilibrium}
					\end{aligned}$} & non-informative equilibrium & non-informative equilibrium \\ \hline
			\end{tabular}
		\end{adjustbox}
	\end{table*}
	\label{thm:nashCases}
\end{thm}
\begin{proof}
Let the transmitter choose any signals $\mathcal{S}=\{S_0,S_1\}$. Assuming nonzero priors $\pi_0^t, \pi_0^r, \pi_1^t$ and $\pi_1^r$, the optimal decision for the receiver is given by \eqref{eq:receiverLRT}. By applying the same extreme case analysis as in the proof of Theorem~\ref{thm:teamCase}, the equilibrium is non-informative if $\tau\leq0$ or $\tau=\infty$ (see Table~\ref{table:receiverCases}); thus, $0<\tau<\infty$ can be assumed. 

Now assume that the receiver applies a single-threshold rule; i.e., $\delta : \Bigg\{ a y \overset{\mathcal{H}_1}{\underset{\mathcal{H}_0}{\gtreqless}}\eta$ where $a\in\mathbb{R}$ and $\eta\in\mathbb{R}$. 
\begin{remark}\label{rem:unstableNash}
Note that for $a=0$, the receiver chooses either always $\mathcal{H}_0$ or always $\mathcal{H}_1$ without considering the value of $y$, which implies a non-informative equilibrium. Therefore, $S_0^*=S_1^*$, $a^*=0$, and $\eta^*=\zeta(\tau-1)$ (i.e., the decision rule of the receiver is $\delta^* : \zeta\overset{\mathcal{H}_1}{\underset{\mathcal{H}_0}{\gtreqless}} \zeta\tau$) constitute a non-informative equilibrium regardless of the values of the priors and costs of the players.
\end{remark} 
Thus, due to the remark above, it can be assumed that $a\neq0$ holds. Since $aY \sim\mathcal{N} \Big(a S_i,a^2\sigma^2\Big)$ under $\mathcal{H}_i$ for $i\in\{0,1\}$,
the conditional probabilities are $\mathsf{P}_{10}=\mathcal{Q}\left(\eta-aS_0\over|a|\sigma\right)$ and $\mathsf{P}_{01}=\mathcal{Q}\left(-{\eta-aS_1\over|a|\sigma}\right)$. Then, the Bayes risk of the transmitter becomes
\begin{align}
r^t(\mathcal{S},\delta) = \pi_0^t C^t_{00} &+ \pi_1^t C^t_{11} + \pi_0^t (C^t_{10}-C^t_{00})\mathcal{Q}\left(\eta-aS_0\over|a|\sigma\right)+ \pi_1^t (C^t_{01}-C^t_{11})\mathcal{Q}\left(-{\eta-aS_1\over|a|\sigma}\right) \,.
\label{eq:transmitterRiskNash}
\end{align}
Since the power constraints are $|S_0|^2 \leq P_0$ and $|S_1|^2 \leq P_1$, the signals $S_0$ and $S_1$ can be regarded as independent, and the optimum signals $\mathcal{S}=\{S_0,S_1\}$ can be found by analyzing the derivative of the Bayes risk of the transmitter with respect to the signals:
\begin{align*}
{\partial\,r^t(\mathcal{S},\delta)\over\partial\,S_i}&= {\text{sgn}(a)\over\sqrt{2\pi}\sigma}\pi_i^t (C^t_{1i}-C^t_{0i})\exp\left\{-{1\over2}\left(\eta-aS_i\over|a|\sigma\right)^2\right\}\,.
\end{align*}
Then, for $i\in\{0,1\}$, the following cases hold:
\begin{enumerate}
	\item \underline{$C^t_{1i}=C^t_{0i}$} $\Rightarrow$ $S_i$ has no effect on the Bayes risk of the transmitter.
	\item \underline{$C^t_{1i}\neq C^t_{0i}$} $\Rightarrow$ $r^t(\mathcal{S},\delta)$ is a decreasing (increasing) function of $S_i$ if $a(C^t_{1i}-C^t_{0i})$ is negative (positive); thus the transmitter chooses the optimal signal levels as $S_0=-\text{sgn}(a)\text{sgn}(C^t_{10}-C^t_{00})\sqrt{P_0}$ and $S_1=\text{sgn}(a)\text{sgn}(C^t_{01}-C^t_{11})\sqrt{P_1}$. 
\end{enumerate}

By using the expressions above, the cases can be listed as follows:
\begin{enumerate}
	\item \underline{$\tau\leq0$ or $\tau=\infty$} $\Rightarrow$ The equilibrium is non-informative.
	\item \underline{$C^t_{10}=C^t_{00}$ (and/or $C^t_{01}=C^t_{11}$)} $\Rightarrow$ $S_0$ (and/or $S_1$) has no effect on the Bayes risk of the transmitter; thus it can arbitrarily be chosen by the transmitter. In this case, if the transmitter chooses $S_0=S_1$; i.e., he does not send anything useful to the receiver, and the receiver applies the decision rule $\delta : \zeta \overset{\mathcal{H}_1}{\underset{\mathcal{H}_0}{\gtreqless}} \zeta\tau$; i.e., he only considers the prior information (totally discards the information sent by the transmitter). Therefore, there exists a non-informative equilibrium.
	\item Notice that, since $0<\tau<\infty$ is assumed, $\zeta=\text{sgn}(C^r_{01}-C^r_{11})=\text{sgn}(C^r_{10}-C^r_{00})$ is obtained. Now, assume that the decision rule of the receiver is $\delta : \Bigg\{ a y \overset{\mathcal{H}_1}{\underset{\mathcal{H}_0}{\gtreqless}}\eta$. Then, the transmitter selects $S_0=-\text{sgn}(a)\text{sgn}(C^t_{10}-C^t_{00})\sqrt{P_0}$ and $S_1=\text{sgn}(a)\text{sgn}(C^t_{01}-C^t_{11})\sqrt{P_1}$ as optimal signals, and the decision rule becomes \eqref{eq:receiverLRT}. By combining the best responses of the transmitter and the receiver,
	\begin{align}
	a&=\zeta(S_1-S_0)=\zeta\text{sgn}(a)\left(\text{sgn}(C^t_{01}-C^t_{11})\sqrt{P_1}+\text{sgn}(C^t_{10}-C^t_{00})\sqrt{P_0}\right) \nn\\
	\Rightarrow & \text{sgn}(a) = \zeta\text{sgn}(a) \text{sgn}\left(\text{sgn}(C^t_{01}-C^t_{11})\sqrt{P_1}+\text{sgn}(C^t_{10}-C^t_{00})\sqrt{P_0}\right) \nn\\
	\Rightarrow & \underbrace{\text{sgn}(C^t_{01}-C^t_{11})\over\text{sgn}(C^r_{01}-C^r_{11})}_{=\text{sgn}(\xi_1)}\sqrt{P_1}+\underbrace{\text{sgn}(C^t_{10}-C^t_{00})\over\text{sgn}(C^r_{10}-C^r_{00})}_{=\text{sgn}(\xi_0)}\sqrt{P_0} > 0
	\label{eq:nashRecursion}
	\end{align} 
	is obtained. Here, unless \eqref{eq:nashRecursion} is satisfied, the best responses of the transmitter and the receiver cannot match each other. Then, there are four possible cases:
	\begin{enumerate}
		\item \underline{$\xi_0<0$ and $\xi_1<0$} $\Rightarrow$ \eqref{eq:nashRecursion} cannot be satisfied; thus, the best responses of the transmitter and the receiver do not match each other, which results in the absence of a Nash equilibrium for $a\neq0$. However, as discussed in Remark~\ref{rem:unstableNash}, $S_0^*=S_1^*$, $a^*=0$, and $\eta^*=\zeta(\tau-1)$ always constitute a non-informative equilibrium.
		\item \underline{$\xi_0<0$ and $\xi_1>0$} $\Rightarrow$ \eqref{eq:nashRecursion} is satisfied only when $\sqrt{P_1}>\sqrt{P_0}$. If $\sqrt{P_1}<\sqrt{P_0}$, \eqref{eq:nashRecursion} cannot be satisfied and the best responses of the transmitter and the receiver do not match each other, which results in the absence of a Nash equilibrium for $a\neq0$. However, due to Remark~\ref{rem:unstableNash}, for $a=0$, there always exist non-informative equilibria. If $\sqrt{P_1}=\sqrt{P_0}$ (which implies $S_0=S_1$), then the receiver applies $\delta : \Bigg\{ \zeta \overset{\mathcal{H}_1}{\underset{\mathcal{H}_0}{\gtreqless}}\zeta\tau$ as in Case-2, and the receiver chooses either always $\mathcal{H}_0$ or always $\mathcal{H}_1$. Hence, there exists a non-informative equilibrium; i.e., the transmitter sends dummy signals, and the receiver makes a decision without considering the transmitted signals.   
		\item \underline{$\xi_0>0$ and $\xi_1<0$} $\Rightarrow$ \eqref{eq:nashRecursion} is satisfied only when $\sqrt{P_0}>\sqrt{P_1}$. If $\sqrt{P_0}<\sqrt{P_1}$, \eqref{eq:nashRecursion} cannot be satisfied and the best responses of the transmitter and the receiver do not match each other, which results in the absence of a Nash equilibrium for $a\neq0$. However, due to Remark~\ref{rem:unstableNash}, for $a=0$, there always exist non-informative equilibria. If $\sqrt{P_0}=\sqrt{P_1}$ (which implies $S_0=S_1$), then the receiver applies $\delta : \Bigg\{ \zeta \overset{\mathcal{H}_1}{\underset{\mathcal{H}_0}{\gtreqless}}\zeta\tau$ as in Case-2, and the equilibrium is non-informative.   
		\item \underline{$\xi_0>0$ and $\xi_1>0$} $\Rightarrow$ \eqref{eq:nashRecursion} is always satisfied; thus, the consistency is established, and there exists an informative equilibrium.
	\end{enumerate}
\end{enumerate}
\end{proof}

As it can be deduced from Table~\ref{table:nashSummary}, as the costs related to both hypotheses are aligned\footnote{$\xi_i$ is the indicator that the transmitter and the receiver have similar preferences about hypothesis $\mathcal{H}_i$; i.e., if $\xi_i>0$, then both the transmitter and the receiver aim to transmit and decode the hypothesis $\mathcal{H}_i$ correctly (or incorrectly). If $\xi_i<0$, then the transmitter and the receiver have conflicting goals over hypothesis $\mathcal{H}_i$; i.e., one of them tries to achieve the correct transmission and decoding, whereas the goal of the other player is the opposite.} for the transmitter and the receiver, the Nash equilibrium is informative. If the power limit corresponding to the hypothesis that has aligned costs for the transmitter and receiver is greater than the power limit of the other hypothesis, again, there exists an informative equilibrium. For the other cases, there may exist non-informative equilibrium.

\begin{remark}\label{rem:nashEq}
	\begin{itemize}
		\item[(i)] We emphasize that, under the Nash formulation, while calculating the best responses, the transmitter and the receiver do not need to know the priors and the costs of each other. In particular,
		\begin{itemize}
			\item for a given decision rule of the receiver $\delta : \Bigg\{ a y \overset{\mathcal{H}_1}{\underset{\mathcal{H}_0}{\gtreqless}}\eta$, the best response of the transmitter is $S_0^{\mathrm{BR}}=-\text{sgn}(a)\text{sgn}(C^t_{10}-C^t_{00})\sqrt{P_0}$ and $S_1^{\mathrm{BR}}=\text{sgn}(a)\text{sgn}(C^t_{01}-C^t_{11})\sqrt{P_1}$.
			\item similarly, for a given signal design $S_0$ and $S_1$ of the transmitter, the best response of the receiver is $a^{\mathrm{BR}}=\zeta(S_1-S_0)$ and $\eta^{\mathrm{BR}}=\zeta\left(\sigma^2\ln(\tau)+{(S_1)^2-(S_0)^2\over2}\right)$.
		\end{itemize}
		\item[(ii)] As shown in Theorem~\ref{thm:nashCases}, at the informative Nash equilibrium, the transmitter selects $S_0^*=-\text{sgn}(a^*)\text{sgn}(C^t_{10}-C^t_{00})\sqrt{P_0}$ and $S_1^*=\text{sgn}(a^*)\text{sgn}(C^t_{01}-C^t_{11})\sqrt{P_1}$, and the decision rule of the receiver is $\delta^* : \Bigg\{ a^* y \overset{\mathcal{H}_1}{\underset{\mathcal{H}_0}{\gtreqless}}\eta^*$, where $a^*=\zeta(S_1^*-S_0^*)$ and $\eta^*=\zeta\left(\sigma^2\ln(\tau)+{(S_1^*)^2-(S_0^*)^2\over2}\right)$. Similar to the team and Stackelberg setup analyses, the informative equilibrium is essentially unique in the Nash case, too; i.e., if $(S_0^*,S_1^*,a^*,\eta^*)$ is an equilibrium point, then $(-S_0^*,-S_1^*,-a^*,\eta^*)$ is another equilibrium point, and they both result in the same Bayes risks for the transmitter and the receiver.
		\item[(iii)] For the non-informative equilibrium, as discussed in Remark~\ref{rem:unstableNash}, the optimal strategies of the transmitter and the receiver are determined by $S_0^*=S_1^*$, $a^*=0$, and $\eta^*=\zeta(\tau-1)$; which results in essentially unique equilibria (see Remark~\ref{rem:teamUnique}-(ii)).
	\end{itemize}
\end{remark}

Even though the transmitter and the receiver do not know the private parameters of each other, they can achieve (converge) to an equilibrium. Note that, due to Remark~\ref{rem:nashEq}-(i), for any arbitrary receiver strategy $(a,\eta)$, the best response of the transmitter $(S_0^{\mathrm{BR}},S_1^{\mathrm{BR}})$ is one of the four possibilities: $(\sqrt{P_0},\sqrt{P_1})$, $(-\sqrt{P_0},\sqrt{P_1})$, $(\sqrt{P_0},-\sqrt{P_1})$, or $(-\sqrt{P_0},-\sqrt{P_1})$. Then, the corresponding best responses of the receiver are characterized by $(a^{\mathrm{BR}}_1,\eta^{\mathrm{BR}})$, $(a^{\mathrm{BR}}_2,\eta^{\mathrm{BR}})$, $(-a^{\mathrm{BR}}_2,\eta^{\mathrm{BR}})$, or $(-a^{\mathrm{BR}}_1,\eta^{\mathrm{BR}})$, respectively, where $a^{\mathrm{BR}}_1\triangleq\zeta(\sqrt{P_1}-\sqrt{P_0})$, $a^{\mathrm{BR}}_2\triangleq\zeta(\sqrt{P_1}+\sqrt{P_0})$, and $\eta^{\mathrm{BR}}=\zeta\left(\sigma^2\ln(\tau)+{P_1-P_0\over2}\right)$. By continuing these iterations, the best responses of the transmitter and the receiver can be combined and \eqref{eq:nashRecursion} is obtained. If their private parameters (priors and costs) satisfy the condition of the unique informative equilibrium in Table IV, their best responses match each other, so the best-response dynamics converges to an equilibrium (e.g., $(a,\eta)\rightarrow(\sqrt{P_0},\sqrt{P_1})\rightarrow(a^{\mathrm{BR}}_1,\eta^{\mathrm{BR}})\rightarrow(\sqrt{P_0},\sqrt{P_1})\rightarrow\cdots$). Otherwise, the optimal strategies (best responses) of the transmitter and the receiver oscillate between two best responses; e.g., $(a,\eta)\rightarrow(\sqrt{P_0},\sqrt{P_1})\rightarrow(a^{\mathrm{BR}}_1,\eta^{\mathrm{BR}})\rightarrow(-\sqrt{P_0},-\sqrt{P_1})\rightarrow(-a^{\mathrm{BR}}_1,\eta^{\mathrm{BR}})\rightarrow(\sqrt{P_0},\sqrt{P_1})\rightarrow\cdots$. Then, they deduce that there exist only non-informative equilibria, in which $S_0^*=S_1^*$, $a^*=0$, and $\eta^*=\zeta(\tau-1)$ (see Remark~\ref{rem:nashEq}-(iii)).

Note that, when $a\neq0$, the misalignment between the costs can even induce a scenario, in which there exists no equilibrium. For $a\neq0$, the main reason for the absence of a non-informative (babbling) equilibrium under the Nash assumption is that in the binary signaling game setup, the receiver is forced to make a decision. Using only the prior information, the receiver always chooses one of the hypothesis. By knowing this, the transmitter can manipulate his signaling strategy for his own benefit. However, after this manipulation, the receiver no longer keeps his decision rule the same; namely, the best response of the receiver alters based on the signaling strategy of the transmitter, which entails another change of the best response of the transmitter. Due to such an infinite recursion, the optimal policies of the transmitter and the receiver keep changing, and thus, there does not exist a pure Nash equilibrium unless $a=0$; i.e., due to Remark~\ref{rem:unstableNash}, there always exist non-informative equilibria with $S_0^*=S_1^*$, $a^*=0$, and $\eta^*=\zeta(\tau-1)$.

\subsection{Continuity and Robustness to Perturbations around the Team Setup} \label{sec:nashContinuity}
Similar to that in Section~\ref{sec:ConvRobustStackelberg} for the Stackelberg setup, the effects of small perturbations in priors and costs on equilibrium values around the team setup are investigated for the Nash setup as follows:

Define the perturbation around the team setup as $\vepsilon=\{\epsilon_{\pi0},\epsilon_{\pi1},\epsilon_{00},\epsilon_{01},\epsilon_{10},\epsilon_{11}\}\in\mathbb{R}^6$ such that $\pi_i^t=\pi_i^r+\epsilon_{\pi i}$ and $C_{ji}^t=C_{ji}^r+\epsilon_{ji}$ for $i,j\in\{0,1\}$  (note that the transmitter parameters are perturbed around the receiver parameters which are assumed to be fixed). Then, for $0<\tau<\infty$, at the point of identical priors and costs, small perturbations in priors and costs imply $\xi_0 = {C^r_{10}-C^r_{00}+\epsilon_{10}-\epsilon_{00}\over C^r_{10}-C^r_{00}}$ and $\xi_1 = {C^r_{01}-C^r_{11}+\epsilon_{01}-\epsilon_{11}\over C^r_{01}-C^r_{11}}$. As it can be seen, the Nash equilibrium is not affected by small perturbations in priors. Further, since $\xi_0=\xi_1=1$ at the point of identical priors and costs for $0<\tau<\infty$, as long as the perturbation $\vepsilon$ is chosen such that $\Big\lvert{\epsilon_{10}-\epsilon_{00}\over C^r_{10}-C^r_{00}}\Big\rvert<1$ and $\Big\lvert{\epsilon_{01}-\epsilon_{11}\over C^r_{01}-C^r_{11}}\Big\rvert<1$, we always obtain positive $\xi_0$ and $\xi_1$ in Table~\ref{table:nashSummary}. Thus, under the Nash assumption, the equilibrium behavior is robust to small perturbations in both priors and costs.  

For the continuity analysis, first consider a non-informative equilibrium; i.e., the policies are $S_0^*=S_1^*$, $a^*=0$, and $\eta^*=\zeta(\tau-1)$, which are independent of the values of the priors and costs of the players. Thus, consider when $a\neq0$; i.e., an informative equilibrium: if the priors and costs are perturbed around the team setup, $S_0=-\text{sgn}(a)\text{sgn}(C^r_{10}-C^r_{00}+\epsilon_{10}-\epsilon_{00})\sqrt{P_0}$ and $S_1=\text{sgn}(a)\text{sgn}(C^r_{01}-C^r_{11}+\epsilon_{01}-\epsilon_{11})\sqrt{P_1}$ are obtained. As long as the perturbation $\vepsilon$ is chosen such that $\Big\lvert{\epsilon_{10}-\epsilon_{00}\over C^r_{10}-C^r_{00}}\Big\rvert<1$ and $\Big\lvert{\epsilon_{01}-\epsilon_{11}\over C^r_{01}-C^r_{11}}\Big\rvert<1$, the changes in $\eta$, $S_0$ and $S_1$ are continuous with respect to perturbations; actually, the values of the equilibrium parameters remain constant; i.e., either $(S_0^*,S_1^*,a^*,\eta^*)=\left(-\zeta \sqrt{P_0},\zeta \sqrt{P_1}, (\sqrt{P_0}+\sqrt{P_1}),\zeta\left(\sigma^2\ln(\tau)+{S_1^2-S_0^2\over2}\right)\right)$ or the essentially equivalent one $(S_0^*,S_1^*,a^*,\eta^*)=\left(\zeta \sqrt{P_0},-\zeta \sqrt{P_1}, -(\sqrt{P_0}+\sqrt{P_1}),\zeta\left(\sigma^2\ln(\tau)+{S_1^2-S_0^2\over2}\right)\right)$ holds.  Thus, the policies are continuous with respect to small perturbations around the point of identical priors and costs.

\subsection{Application to the Motivating Examples}
\subsubsection{Subjective Priors}
The related parameters are $\tau={\pi_0^r (C_{10}-C_{00}) \over \pi_1^r (C_{01}-C_{11})}$, $\xi_0 =1$, and $\xi_1=1$. Thus, if $\tau<0$ or $\tau=\infty$, the equilibrium is non-informative; otherwise, there always exists a unique informative equilibrium. 

\subsubsection{Biased Transmitter Cost}
Based on the arguments in Section~\ref{sec:biasedCost}, the related parameters can be found as follows:
	\begin{align*}
	C^t_{01}&=C^t_{10}=\alpha \text{ and } C^t_{00}=C^t_{11}=1-\alpha \,,\\
	C^r_{01}&=C^r_{10}=1 \text{ and } C^r_{00}=C^r_{11}=0 \,,\\
	\tau&={\pi_0 (C^r_{10}-C^r_{00}) \over \pi_1 (C^r_{01}-C^r_{11})}={\pi_0\over\pi_1} \,,\\
	\xi_0 &= {C^t_{10}-C^t_{00}\over C^r_{10}-C^r_{00}} =2\alpha-1 \,,\\
	\xi_1 &= {C^t_{01}-C^t_{11}\over C^r_{01}-C^r_{11}}=2\alpha-1\,.
	\end{align*}
	If $\alpha>1/2$ (Case-3-d of Theorem~\ref{thm:nashCases} applies), the players act like a team and the equilibrium is informative. If $\alpha=1/2$ (Case-2 of Theorem~\ref{thm:nashCases} applies), the equilibrium is non-informative. Otherwise; i.e., if $\alpha<1/2$ (Case-3-a of Theorem~\ref{thm:nashCases} applies), there exist non-informative equilibria. As it can be seen, the existence of the equilibrium depends on $\alpha=\mathsf{Pr}(b=0)$, the probability that the Bayes risks of the transmitter and the receiver are aligned.

\section{EXTENSION to the MULTI-DIMENSIONAL CASE}

When the transmitter sends a multi-dimensional signal over a multi-dimensional channel, or the receiver takes multiple samples from the observed waveform, the scalar analysis considered heretofore is not applicable anymore; thus, the vector case can be investigated. In this direction, the binary hypothesis-testing problem aforementioned can be modified as
\begin{align*}
\mathcal{H}_0 : \vY = \vS_0 + \vN \;,\nn\\
\mathcal{H}_1 : \vY = \vS_1 + \vN \;,
\end{align*}
where $\vY$ is the observation (measurement) vector that belongs to the observation set $\Gamma=\mathbb{R}^n$, $\vS_0$ and $\vS_1$ denote the deterministic signals under hypothesis $\mathcal{H}_0$ and hypothesis $\mathcal{H}_1$, such that $\mathbb{S}\triangleq\{\mathcal{S}:\lVert \vS_0 \rVert ^2 \leq P_0 \,,\; \lVert \vS_1 \rVert ^2 \leq P_1\}$, respectively, and $\vN$ represents a zero-mean Gaussian noise vector with the positive definite covariance matrix $\Sigma$; i.e., $\vN \sim \mathcal{N} (\mathbf{0},\Sigma)$. All the other parameters ($\pi_i^k$ and $C^k_{ji}$ for $i,j\in\{0,1\}$ and $k\in\{t,r\}$) and their definitions remain unchanged.

\subsection{Team Setup Analysis} \label{sec:vecTeam}

\begin{thm}
	Theorem~\ref{thm:teamCase} also holds for the vector case: if $0<\tau<\infty$, the team solution is always informative; otherwise, there exist only non-informative equilibria.
	\label{thm:teamCaseVector}
\end{thm}
\begin{proof}
Let the transmitter choose optimal signals $\mathcal{S}=\{\vS_0,\vS_1\}$. Then the measurements become $\mathcal{H}_i : \vY \sim \mathcal{N} (\vS_i,\Sigma)$ for $i\in\{0,1\}$. As in the scalar case in Theorem~\ref{thm:teamCase}, the equilibrium is non-informative for $\tau\leq0$ or $\tau=\infty$; hence, $0<\tau<\infty$ can be assumed. Similar to \eqref{eq:receiverLRT}, the optimal decision rule for the receiver is obtained by utilizing \eqref{eq:lrt} as
\begin{align}
\delta^*_{\vS_0,\vS_1} &: \Bigg\{ \zeta{ p_1(\vy)\over p_0(\vy)} \overset{\mathcal{H}_1}{\underset{\mathcal{H}_0}{\gtreqless}} \zeta{\pi_0^r (C^r_{10}-C^r_{00}) \over \pi_1^r (C^r_{01}-C^r_{11})} \triangleq \zeta\tau \nn\\
&:\Bigg\{ \zeta{ {1\over\sqrt{(2\pi)^n|\Sigma|}}\exp\left\{-{1\over 2}(\vy-\vS_1)^T\Sigma^{-1}(\vy-\vS_1)\right\} \over {1\over\sqrt{(2\pi)^n|\Sigma|}}\exp\left\{-{1\over 2}(\vy-\vS_0)^T\Sigma^{-1}(\vy-\vS_0)\right\}} \overset{\mathcal{H}_1}{\underset{\mathcal{H}_0}{\gtreqless}} \zeta\tau \nn\\
\begin{split}
&:\Bigg\{ \zeta(\vS_1-\vS_0)^T\Sigma^{-1}\vy \overset{\mathcal{H}_1}{\underset{\mathcal{H}_0}{\gtreqless}} \zeta\left(\ln(\tau)+{1\over2}(\vS_1-\vS_0)^T\Sigma^{-1}(\vS_1+\vS_0)\right)\;.
\label{eq:receiverLRTVector}
\end{split}
\end{align}
Since, under hypothesis $\mathcal{H}_i$, $\zeta(\vS_1-\vS_0)^T\Sigma^{-1}\vY\sim \mathcal{N} \left(\zeta(\vS_1-\vS_0)^T\Sigma^{-1}\vS_i,(\vS_1-\vS_0)^T\Sigma^{-1}(\vS_1-\vS_0)\right)$ for $i\in\{0,1\}$, by defining $d^2\triangleq(\vS_1-\vS_0)^T\Sigma^{-1}(\vS_1-\vS_0)$, the conditional probabilities can be written as follows:
	\begin{align}
	\mathsf{P}_{10}&=\mathcal{Q}\left(\zeta{\ln(\tau)+{1\over2}(\vS_1-\vS_0)^T\Sigma^{-1}(\vS_1+\vS_0-2\vS_0)\over\sqrt{(\vS_1-\vS_0)^T\Sigma^{-1}(\vS_1-\vS_0)}}\right)=\mathcal{Q}\left(\zeta\left({\ln(\tau)\over d}+{d\over2}\right)\right)\,, \nn\\
	\mathsf{P}_{01}&=1-\mathcal{Q}\left(\zeta{\ln(\tau)+{1\over2}(\vS_1-\vS_0)^T\Sigma^{-1}(\vS_1+\vS_0-2\vS_1)\over\sqrt{(\vS_1-\vS_0)^T\Sigma^{-1}(\vS_1-\vS_0)}}\right)=1-\mathcal{Q}\left(\zeta\left({\ln(\tau) \over d} -{d\over2}\right)\right)\nn\\
	&=\mathcal{Q}\left(\zeta\left(-{\ln(\tau)\over d}+{d\over2}\right)\right)\;.
	\label{eq:condProbsVector}
	\end{align}
	Notice that the conditional probabilities are the same in \eqref{eq:condProbs} and \eqref{eq:condProbsVector}; therefore, in the vector case, the equilibrium is always informative, and the transmitter always prefers the maximum distance similar to the scalar case. However, selecting optimal vector signals is not as trivial as in the scalar case; see \cite[pp. 61--63]{Poor} for details. Since the eigenvector with the largest (smallest) eigenvalue of $\Sigma$ corresponds to the direction, along which the noise is most (least) powerful, signaling in the least noisy direction results in the highest signal-to-noise power ratio for the system. Accordingly, the optimum signals are $\vS_0 = \pm \sqrt{P_0} {\vnu_{\min}\over\|\vnu_{\min}\|}$ and $\vS_1 = \mp \sqrt{P_1} {\vnu_{\min}\over\|\vnu_{\min}\|}$, which corresponds to $d_{\max}^2={(\sqrt{P_0}+\sqrt{P_1})^2\over\lambda_{\min}}$, where $\lambda_{\min}$ is the minimum eigenvalue of $\Sigma$ and $\vnu_{\min}$ is the eigenvector corresponding to $\lambda_{\min}$ \cite[pp. 61--63]{Poor}.
\end{proof}
	
\subsection{Stackelberg Game Analysis}
\begin{thm}
	Let $d\triangleq\sqrt{(\vS_1-\vS_0)^T\Sigma^{-1}(\vS_1-\vS_0)}$ and $d_{\max}^2\triangleq{(\sqrt{P_0}+\sqrt{P_1})^2\over\lambda_{\min}}$, where $\lambda_{\min}$ is the minimum eigenvalue of $\Sigma$. Then Theorem~\ref{thm:stackelbergCases} also holds for the vector case.
	\label{thm:stackelbergCasesVector}
\end{thm}
\begin{proof}
The proof of Theorem~\ref{thm:stackelbergCases} can be applied by modifying the definitions of $d$ and $d_{\max}$ as in the statement. For $d^*=d_{\max}$, the method described in the proof of Theorem~\ref{thm:teamCaseVector} can be applied for the optimal signal selection, whereas, for $d^*=0$, by choosing $\vS_0=\vS_1$, the non-informative equilibrium can be achieved. Further, for Case-3 of Theorem~\ref{thm:stackelbergCases}, in order to achieve $(d^*)^2=\Big|{2\ln\tau(k_0-k_1)\over(k_0+k_1)}\Big|<d_{\max}^2$, the signals can be chosen in the direction of $\vnu_{\min}$, that is, the eigenvector corresponding to $\lambda_{\min}$. Accordingly, $\vS_0=(-\sqrt{P_0}+t) {\vnu_{\min}\over\|\vnu_{\min}\|}$ and $\vS_1=(-\sqrt{P_0}+d^*+t) {\vnu_{\min}\over\|\vnu_{\min}\|}$ for $t\in[0,\sqrt{P_1}+\sqrt{P_0}-d^*]$ are possible optimal signal pairs. Similarly, $\vS_0=(\sqrt{P_0}-t) {\vnu_{\min}\over\|\vnu_{\min}\|}$ and $\vS_1=(\sqrt{P_0}-d^*-t) {\vnu_{\min}\over\|\vnu_{\min}\|}$ for $t\in[0,\sqrt{P_1}+\sqrt{P_0}-d^*]$ consist of another set of possible optimal signal pairs. Note that it may be possible to find optimal signal pairs $\{\vS_0,\vS_1\}\in\mathbb{S}$ that satisfy $(\vS_1-\vS_0)^T\Sigma^{-1}(\vS_1-\vS_0)=\Big|{2\ln\tau(k_0-k_1)\over(k_0+k_1)}\Big|$ in any other direction rather than the direction of $\vnu_{\min}$; however, finding a single pair that corresponds to an equilibrium would be sufficient.
\end{proof}

\subsection{Nash Game Analysis}

\begin{thm}
	Theorem~\ref{thm:nashCases} also holds for the vector case.
	\label{thm:NashCaseVector}
\end{thm}
\begin{proof}
	Let the transmitter choose any signals $\mathcal{S}=\{\vS_0,\vS_1\}$. Assuming nonzero priors $\pi_0^t, \pi_0^r, \pi_1^t$ and $\pi_1^r$, the optimal decision rule for the receiver is given by \eqref{eq:receiverLRTVector}. Similar to the team case analysis in Section~\ref{sec:vecTeam}, the equilibrium is non-informative if $\tau\leq0$ or $\tau=\infty$; thus, $0<\tau<\infty$ can be assumed. 

Now assume that the receiver applies a single-threshold rule; i.e., $\delta : \Bigg\{ \va^T \vy \overset{\mathcal{H}_1}{\underset{\mathcal{H}_0}{\gtreqless}}\eta$ where $\va\in\mathbb{R}^n$ and $\eta\in\mathbb{R}$.
\begin{remark}\label{rem:unstableNashVec}
	Note that for $\va=\mathbf{0}$, the receiver chooses either always $\mathcal{H}_0$ or always $\mathcal{H}_1$ without considering the value of $\vy$, which implies a non-informative equilibrium. Therefore, $\vS_0^*=\vS_1^*$, $\va^*=\mathbf{0}$, and $\eta^*=\zeta(\tau-1)$ (i.e., the decision rule of the receiver is $\delta^* : \zeta\overset{\mathcal{H}_1}{\underset{\mathcal{H}_0}{\gtreqless}} \zeta\tau$) constitute a non-informative equilibrium regardless of the values of the priors and costs of the players.
\end{remark} 
Thus, due to the remark above, it can be assumed that $\va\neq\mathbf{0}$ holds. Since $\va^T \vY \sim\mathcal{N} \Big(\va^T \vS_i,\va^T\Sigma\va\Big)$ under $\mathcal{H}_i$ for $i\in\{0,1\}$,
	the conditional probabilities are $\mathsf{P}_{10}=\mathcal{Q}\left(\eta-\va^T \vS_0\over\sqrt{\va^T\Sigma\va}\right)$ and $\mathsf{P}_{01}=\mathcal{Q}\left(-{\eta-\va^T \vS_1\over\sqrt{\va^T\Sigma\va}}\right)$.
	Then, the Bayes risk of the transmitter becomes
	\begin{align*}
	r^t(\mathcal{S},\delta) = \pi_0^t C^t_{00} + \pi_1^t C^t_{11} &+ \pi_0^t (C^t_{10}-C^t_{00})\mathcal{Q}\left(\eta-\va^T \vS_0\over\sqrt{\va^T\Sigma\va}\right) + \pi_1^t (C^t_{01}-C^t_{11})\mathcal{Q}\left(-{\eta-\va^T \vS_1\over\sqrt{\va^T\Sigma\va}}\right) \,.
	\end{align*}
	Since the power constraints are $\lVert \vS_0 \rVert ^2 \leq P_0$ and $\lVert \vS_1 \rVert ^2 \leq P_1$, the signals $\vS_0$ and $\vS_1$ can be regarded as independent. Since $\mathcal{Q}$ function is a monotone decreasing, the following cases hold for $i\in\{0,1\}$:
\begin{enumerate}
	\item \underline{$C^t_{1i}<C^t_{0i}$} $\Rightarrow$ Then, $r^t(\mathcal{S},\delta)$ is a decreasing function of $\va^T \vS_i$, thus the transmitter always chooses $\va^T \vS_i$ as maximum subject to $\lVert \vS_i \rVert ^2 \leq P_i$; i.e., $\vS_i=\sqrt{P_i}{\va\over\|\va\|}$. 
	\item \underline{$C^t_{1i}=C^t_{0i}$} $\Rightarrow$ Then $\vS_i$ has no effect on the Bayes risk of the transmitter.
	\item \underline{$C^t_{1i}>C^t_{0i}$} $\Rightarrow$ Then, $r^t(\mathcal{S},\delta)$ is an increasing function of $\va^T \vS_i$, thus the transmitter always chooses $\va^T \vS_i$ as minimum subject to $\lVert \vS_i \rVert ^2 \leq P_i$; i.e., $\vS_i=-\sqrt{P_i}{\va\over\|\va\|}$.
\end{enumerate}
Thus, the the optimal signals can be characterized as $S_0=-\text{sgn}(C^t_{10}-C^t_{00})\sqrt{P_0}{\va\over\|\va\|}$ and $S_1=\text{sgn}(C^t_{01}-C^t_{11})\sqrt{P_1}{\va\over\|\va\|}$.
	
By using the expressions above, the cases can be listed as follows:
\begin{enumerate}
	\item \underline{$\tau\leq0$ or $\tau=\infty$} $\Rightarrow$ The equilibrium is non-informative.
	\item \underline{$C^t_{10}=C^t_{00}$ (and/or $C^t_{01}=C^t_{11}$)} $\Rightarrow$ $\vS_0$ (and/or $\vS_1$) has no effect on the Bayes risk of the transmitter, thus it can arbitrarily be chosen by the transmitter. In this case, if the transmitter chooses $\vS_0=\vS_1$; i.e., he does not send anything useful to the receiver, and the receiver applies the decision rule $\delta : \zeta \overset{\mathcal{H}_1}{\underset{\mathcal{H}_0}{\gtreqless}} \zeta\tau$; i.e., he only considers the prior information (totally discards the information sent by the transmitter). Then there exists a non-informative equilibrium.
\item Notice that, since $0<\tau<\infty$ is assumed, $\zeta=\text{sgn}(C^r_{01}-C^r_{11})=\text{sgn}(C^r_{10}-C^r_{00})$ is obtained. Now, assume that the decision rule of the receiver is $\delta : \Bigg\{ \va^T \vy \overset{\mathcal{H}_1}{\underset{\mathcal{H}_0}{\gtreqless}}\eta$. Then, the transmitter selects $S_0=-\text{sgn}(C^t_{10}-C^t_{00})\sqrt{P_0}{\va\over\|\va\|}$ and $S_1=\text{sgn}(C^t_{01}-C^t_{11})\sqrt{P_1}{\va\over\|\va\|}$ as optimal signals, and the decision rule becomes \eqref{eq:receiverLRTVector}. By combining the best responses of the transmitter and the receiver,
	\begin{align}
\va^T&=\zeta(\vS_1-\vS_0)^T\Sigma^{-1}=\zeta{\va^T\over\|\va\|}\left(\text{sgn}(C^t_{01}-C^t_{11})\sqrt{P_1}+\text{sgn}(C^t_{10}-C^t_{00})\sqrt{P_0}\right)\Sigma^{-1} \nn\\
\Rightarrow & \va^T\va = {\va^T\Sigma^{-1}\va\over\|\va\|} \zeta\left(\text{sgn}(C^t_{01}-C^t_{11})\sqrt{P_1}+\text{sgn}(C^t_{10}-C^t_{00})\sqrt{P_0}\right) \nn\\
\Rightarrow & \underbrace{\text{sgn}(C^t_{01}-C^t_{11})\over\text{sgn}(C^r_{01}-C^r_{11})}_{=\text{sgn}(\xi_1)}\sqrt{P_1}+\underbrace{\text{sgn}(C^t_{10}-C^t_{00})\over\text{sgn}(C^r_{10}-C^r_{00})}_{=\text{sgn}(\xi_0)}\sqrt{P_0} > 0 \,.
\label{eq:nashVectorRecursion}
\end{align} 
Notice that the expressions in \eqref{eq:nashVectorRecursion} and \eqref{eq:nashRecursion} of Theorem~\ref{thm:nashCases} are the same, and Remark~\ref{rem:unstableNash} and Remark~\ref{rem:unstableNashVec} are equivalent; hence, the Nash equilibrium solution of Theorem~\ref{thm:nashCases} also holds for the vector case. 
	\end{enumerate}
\end{proof}

\section{EXTENSION TO a SCENARIO with an AVERAGE POWER CONSTRAINT}

Besides the peak power constraint considered in the previous sections, the average power constraint can be assumed at the transmitter side. Before presenting the technical results, we provide the following lemma which will be utilized in the equilibrium analyses of the team and Stackelberg setups.

\begin{lemma}\label{lem:maxDistance}
The optimal solutions to the optimization problem
\begin{align}
	\underset{{S_0,S_1}}{\text{sup}} ~ (S_1-S_0)^2 ~~~
\text{s.t.} ~~ \beta_0S_0^2+\beta_1S_1^2 \leq P, ~~ \beta_0,\beta_1\in\mathbb{R}_{>0}
\label{eq:maxProblemClosedForm}
\end{align}
are $(S_0^*,S_1^*)=\left(-\sqrt{{\beta_1\over\beta_0(\beta_0+\beta_1)}P},\sqrt{{\beta_0\over\beta_1(\beta_0+\beta_1)}P}\right)$ and $(S_0^*,S_1^*)=\left(\sqrt{{\beta_1\over\beta_0(\beta_0+\beta_1)}P},-\sqrt{{\beta_0\over\beta_1(\beta_0+\beta_1)}P}\right)$.
\end{lemma}
\begin{proof}
	Observe the following inequalities:
\begin{align}
\beta_0\beta_1(S_1-S_0)^2&=\beta_0\beta_1\left(S_1^2-2S_1S_0+S_0^2\right)\overset{(a)}{\leq} \beta_0\beta_1\left(S_1^2+2|S_1||S_0|+S_0^2\right) \nn\\
&= \beta_0\beta_1\left(S_1^2+S_0^2\right)+2|\beta_0S_0||\beta_1S_1|\overset{(b)}{\leq} \beta_0\beta_1\left(S_1^2+S_0^2\right)+\beta_0^2S_0^2+\beta_1^2S_1^2 \nn\\
&= \beta_0\left(\beta_0S_0^2+\beta_1S_1^2\right) + \beta_1\left(\beta_0S_0^2+\beta_1S_1^2\right)\overset{(c)}{\leq} (\beta_0+\beta_1)P \,.
\end{align}
Here, (b) follows from the inequality for the arithmetic and geometric mean, and the equality holds iff $\beta_1^2S_1^2=\beta_0^2S_0^2$. For (a), the equality holds iff $S_1S_0\leq0$; and for (c), the equality holds iff $\beta_0S_0^2+\beta_1S_1^2=P$. Thus, the upper bound of $(S_1-S_0)^2$ can be achieved with optimal solutions $(S_0^*,S_1^*)=\left(-\sqrt{{\beta_1\over\beta_0(\beta_0+\beta_1)}P},\sqrt{{\beta_0\over\beta_1(\beta_0+\beta_1)}P}\right)$ or $(S_0^*,S_1^*)=\left(\sqrt{{\beta_1\over\beta_0(\beta_0+\beta_1)}P},-\sqrt{{\beta_0\over\beta_1(\beta_0+\beta_1)}P}\right)$ so that $(S_1^*-S_0^*)^2={\beta_0+\beta_1\over\beta_0\beta_1}P$.
\end{proof}	

Consider a transmitter with an average power constraint; i.e., the transmitter performs the optimal signal design problem under the power constraint below:
\begin{align*}
\mathbb{S}\triangleq\{\mathcal{S}=\{S_0,S_1\}:\pi_0^t| S_0 | ^2 + \pi_1^t| S_1 | ^2\leq P_{\mathrm{avg}}\} \;,
\end{align*}
where $P_{\mathrm{avg}}$ denotes the average power limit.

\subsection{Team Theoretic Analysis}
In order to minimize the Bayes risk, the transmitter always prefers the maximum $d={|S_1-S_0|\over\sigma}$. Thus, by Lemma~\ref{lem:maxDistance}, the optimal signal levels are chosen as either $(S_0^*,S_1^*)=\left(-\sqrt{{\pi_1^t\over\pi_0^t(\pi_0^t+\pi_1^t)}P_{\mathrm{avg}}},\sqrt{{\pi_0^t\over\pi_1^t(\pi_0^t+\pi_1^t)}P_{\mathrm{avg}}}\right)$ or $(S_0^*,S_1^*)=\left(\sqrt{{\pi_1^t\over\pi_0^t(\pi_0^t+\pi_1^t)}P_{\mathrm{avg}}},-\sqrt{{\pi_0^t\over\pi_1^t(\pi_0^t+\pi_1^t)}P_{\mathrm{avg}}}\right)$. The corresponding optimal decision rule of the receiver is chosen based on the rule in \eqref{eq:receiverLRT} accordingly. Actually, the equilibrium points are essentially unique; i.e., they result in the same Bayes risks for the transmitter and the receiver.

\subsection{Stackelberg Game Analysis}
Similar to the team setup analysis, for every possible case in Table~\ref{table:stackelbergSummary}, there are more than one equilibrium points, and they are essentially unique since the Bayes risks of the transmitter and the receiver depend on $d$. For example, for $d^*=d_{\max}\triangleq{\sqrt{{\pi_0^t+\pi_1^t\over\pi_0^t\pi_1^t}P_{\mathrm{avg}}}\over\sigma}$, $(S_0^*,S_1^*)=\left(-\sqrt{{\pi_1^t\over\pi_0^t(\pi_0^t+\pi_1^t)}P_{\mathrm{avg}}},\sqrt{{\pi_0^t\over\pi_1^t(\pi_0^t+\pi_1^t)}P_{\mathrm{avg}}}\right)$ and $(S_0^*,S_1^*)=\left(\sqrt{{\pi_1^t\over\pi_0^t(\pi_0^t+\pi_1^t)}P_{\mathrm{avg}}},-\sqrt{{\pi_0^t\over\pi_1^t(\pi_0^t+\pi_1^t)}P_{\mathrm{avg}}}\right)$ are the only possible choices for the transmitter, and the decision rule of the receiver is chosen based on the rule in \eqref{eq:receiverLRT}. However, for $d^*=0$, there are infinitely many choices for the transmitter and the receiver, and all of them are essentially unique; i.e., they result in the same Bayes risks for the transmitter and the receiver. A similar argument holds for $d^*=\sqrt{\Big|{2\ln\tau(k_0-k_1)\over(k_0+k_1)}\Big|}$; i.e., there are infinitely many choices for the transmitter and the receiver, and all of them are essentially unique.

\subsection{Nash Game Analysis}
For $0<\tau<\infty$, if the receiver applies a single-threshold rule\footnote{Due to Remark~\ref{rem:unstableNash}, $S_0^*=S_1^*$, $a^*=0$, and $\eta^*=\zeta(\tau-1)$ always constitute a non-informative equilibrium regardless of the values of the priors and costs of the players}; i.e., $\delta : \Bigg\{ a y \overset{\mathcal{H}_1}{\underset{\mathcal{H}_0}{\gtreqless}}\eta$ where $a\in\mathbb{R}-\{0\}$, and $\eta\in\mathbb{R}$, after analyzing the derivative of the Bayes risk of the transmitter in \eqref{eq:transmitterRiskNash} with respect to the signals, the following can be obtained:
	\begin{enumerate}
		\item \underline{$C^t_{1i}=C^t_{0i}$} $\Rightarrow$ $S_i$ has no effect on the Bayes risk of the transmitter.
		\item \underline{$C^t_{1i}<C^t_{0i}$ or $C^t_{1i}>C^t_{0i}$} $\Rightarrow$ $r^t(\mathcal{S},\delta)$ is a decreasing (increasing) function of $S_i$ if $a(C^t_{1i}-C^t_{0i})$ is negative (positive); thus the transmitter chooses the optimal signal level $S_i$ as large as possible in absolute value. Therefore, the transmitter prefers to utilize the maximum possible total power; i.e., the power constraint can be considered as $\pi_1^tS_1^2+\pi_0^tS_0^2=P_{\mathrm{avg}}$ rather than $\pi_1^tS_1^2+\pi_0^tS_0^2\leq P_{\mathrm{avg}}$.
	\end{enumerate}
	By using the analysis above, the cases can be listed as follows:
	\begin{enumerate}
		\item \underline{$C^t_{1i}=C^t_{0i}$} $\Rightarrow$ If $C^t_{1j}=C^t_{0j}$ also holds for $j\neq i$, then neither $S_0$ nor $S_1$ changes the Bayes risk of the transmitter; thus, there exists a non-informative equilibrium. Otherwise; i.e., $C^t_{1j}\neq C^t_{0j}$ for $j\neq i$, the transmitter chooses the optimal signal levels as $S_i=0$ and $S_j=-\text{sgn}\left(a(C^t_{1j}-C^t_{0j})\right)\sqrt{P_{\mathrm{avg}}\over\pi_j^t}$, and the equilibrium is informative.
		\item \underline{$C^t_{10}\neq C^t_{00}$ and $C^t_{11}\neq C^t_{01}$} $\Rightarrow$ Since the transmitter adjust the signal levels such that $\pi_1^tS_1^2+\pi_0^tS_0^2=P$, the optimal signals must be in the form of $S_0=-\text{sgn}\big(a(C^t_{10}-C^t_{00})\big)x$ and $S_1=\text{sgn}\big(a(C^t_{01}-C^t_{11})\big)\sqrt{P_{\mathrm{avg}}-\pi_0^tx^2\over\pi_1^t}$ for $x\in\left[0,\sqrt{P_{\mathrm{avg}}\over\pi_0^t}\right]$. Then, the Bayes risk of the transmitter in \eqref{eq:transmitterRiskNash} can be expressed as
		\begin{align}
		\begin{split}\label{eq:averageNashRisk}
		r^t(\mathcal{S},\delta) &= \pi_0^t C^t_{00} + \pi_1^t C^t_{11} + \pi_0^t (C^t_{10}-C^t_{00})\mathcal{Q}\left(\eta+|a|\text{sgn}\left(C^t_{10}-C^t_{00}\right)x\over|a|\sigma\right)\\
		&\qquad+ \pi_1^t (C^t_{01}-C^t_{11})\mathcal{Q}\left(-{\eta-|a|\text{sgn}\left(C^t_{01}-C^t_{11}\right)\sqrt{P_{\mathrm{avg}}-\pi_0^tx^2\over\pi_1^t}\over|a|\sigma}\right) \,.
		\end{split}
		\end{align}
		Note that the convexity of $r^t(\mathcal{S},\delta)$ in \eqref{eq:averageNashRisk} with respect to $x$ changes depending on the other parameters (i.e., priors, costs and the receiver policy); hence, the optimal $x$ cannot be expressed in a closed form. Let $x^*$ be an optimal solution to \eqref{eq:averageNashRisk}; i.e., $x^*=\arg\min_{x\in\left[0,\sqrt{P_{\mathrm{avg}}\over\pi_0^t}\right]} r^t(\mathcal{S},\delta)$, which implies that the optimal signal levels are $S_0=-\text{sgn}\big(a(C^t_{10}-C^t_{00})\big)x^*$ and $S_1=\text{sgn}\big(a(C^t_{01}-C^t_{11})\big)\sqrt{P_{\mathrm{avg}}-\pi_0^t(x^*)^2\over\pi_1^t}$. Then, similar to \eqref{eq:nashRecursion}, the following condition on the existence of an equilibrium can be obtained:
		\begin{align}
		\underbrace{\text{sgn}(C^t_{01}-C^t_{11})\over\text{sgn}(C^r_{01}-C^r_{11})}_{=\text{sgn}(\xi_1)}\sqrt{P_{\mathrm{avg}}-\pi_0^t(x^*)^2\over\pi_1^t}+\underbrace{\text{sgn}(C^t_{10}-C^t_{00})\over\text{sgn}(C^r_{10}-C^r_{00})}_{=\text{sgn}(\xi_0)}x^* > 0 \,.
		\label{eq:nashRecursionAvg}
		\end{align} 
	 	Here, similar to the analysis under the individual power constraint in Theorem~\ref{thm:nashCases}, unless \eqref{eq:nashRecursionAvg} is satisfied, the best responses of the transmitter and the receiver cannot match each other. In particular, 
	 		\begin{enumerate}
	 		\item \underline{$\xi_0<0$ and $\xi_1<0$} $\Rightarrow$ There does not exist a Nash equilibrium for $a\neq0$; however, due to Remark~\ref{rem:unstableNash}, for $a=0$, there always exist non-informative equilibria.
	 		\item \underline{$\xi_0<0$ and $\xi_1>0$} $\Rightarrow$ If $\sqrt{P_{\mathrm{avg}}-\pi_0^t(x^*)^2\over\pi_1^t}>x^* \Rightarrow x^* < \sqrt{P_{\mathrm{avg}}}$, then the Nash equilibrium is informative. If $x^*=\sqrt{P_{\mathrm{avg}}}$, there exists a non-informative equilibrium. Otherwise; i.e., if $x^*>\sqrt{P_{\mathrm{avg}}}$, there does not exist a Nash equilibrium for $a\neq0$; however, due to Remark~\ref{rem:unstableNash}, for $a=0$, there always exist non-informative equilibria.    
	 		\item \underline{$\xi_0>0$ and $\xi_1<0$} $\Rightarrow$ If $x^* > \sqrt{P_{\mathrm{avg}}}$, then the Nash equilibrium is informative. If $x^*=\sqrt{P_{\mathrm{avg}}}$, there exists a non-informative equilibrium. Otherwise; i.e., if $x^*<\sqrt{P_{\mathrm{avg}}}$, there does not exist a Nash equilibrium for $a\neq0$; however, due to Remark~\ref{rem:unstableNash}, for $a=0$, there always exist non-informative equilibria.     
	 		\item \underline{$\xi_0>0$ and $\xi_1>0$} $\Rightarrow$ There exists an informative Nash equilibrium.
	 	\end{enumerate}
	\end{enumerate}	
%

\section{CONCLUDING REMARKS}

In this paper, we considered binary signaling problems in which the decision makers (the transmitter and the receiver) have subjective priors and/or misaligned objective functions. Depending on the commitment nature of the transmitter to his policies, we formulated the binary signaling problem as a Bayesian game under either Nash or Stackelberg equilibrium concepts and established equilibrium solutions and their properties. 

We showed that there can be informative or non-informative equilibria in the binary signaling game under the Stackelberg and Nash assumptions, and derived the conditions under which an informative equilibrium exists. We also studied the effects of small perturbations around the team setup (with identical priors and costs) 
and showed that the game equilibrium behavior around the team setup is robust under the Nash assumption, whereas it is not robust under the Stackelberg assumption.

The binary setup considered here can be extended to the $M$-ary hypothesis testing setup, and the corresponding signaling game structure can be formed in order to model a game between players with a multiple-bit communication channel. The extension to more general noise distributions is possible: the Nash equilibrium analysis holds identically when the noise distribution leads to a single-threshold test. Finally, in addition to the Bayesian approach considered here, different cost structures and parameters can be introduced by investigating the game under Neyman-Pearson and mini-max criteria. 

\bibliographystyle{IEEEtran}
\bibliography{../SerkanBibliography}

\begin{thebibliography}{10}
\providecommand{\url}[1]{#1}
\csname url@samestyle\endcsname
\providecommand{\newblock}{\relax}
\providecommand{\bibinfo}[2]{#2}
\providecommand{\BIBentrySTDinterwordspacing}{\spaceskip=0pt\relax}
\providecommand{\BIBentryALTinterwordstretchfactor}{4}
\providecommand{\BIBentryALTinterwordspacing}{\spaceskip=\fontdimen2\font plus
\BIBentryALTinterwordstretchfactor\fontdimen3\font minus
  \fontdimen4\font\relax}
\providecommand{\BIBforeignlanguage}[2]{{%
\expandafter\ifx\csname l@#1\endcsname\relax
\typeout{** WARNING: IEEEtran.bst: No hyphenation pattern has been}%
\typeout{** loaded for the language `#1'. Using the pattern for}%
\typeout{** the default language instead.}%
\else
\language=\csname l@#1\endcsname
\fi
#2}}
\providecommand{\BIBdecl}{\relax}
\BIBdecl

\bibitem{cdc2018HT}
S.~Sar{\i}ta\c{s}, S.~Gezici, and S.~Y\"uksel, ``Binary signaling under
  subjective priors and costs as a game,'' in \emph{57th IEEE Conference on
  Decision and Control (CDC)}, Dec. 2018, pp. 1130--1135.

\bibitem{sandberg2015cyberphysical}
H.~Sandberg, S.~S.~Amin, and K.~H. Johansson, ``Cyberphysical security in
  networked control systems: An introduction to the issue,'' \emph{IEEE Control
  Systems}, vol.~35, no.~1, pp. 20--23, 2015.

\bibitem{teixeira2015secure}
A.~Teixeira, I.~Shames, H.~Sandberg, and K.~H. Johansson, ``A secure control
  framework for resource-limited adversaries,'' \emph{Automatica}, vol.~51, pp.
  135--148, 2015.

\bibitem{networkSecurity}
T.~Alpcan and T.~Ba\c{s}ar, \emph{Network Security: A Decision and
  Game-Theoretic Approach}, 1st~ed.\hskip 1em plus 0.5em minus 0.4em\relax New
  York, NY, USA: Cambridge University Press, 2010.

\bibitem{cyberSecuritySmartGrid}
Y.~Mo, T.-H. Kim, K.~Brancik, D.~Dickinson, H.~Lee, A.~Perrig, and B.~Sinopoli,
  ``Cyber physical security of a smart grid infrastructure,'' \emph{Proceedings
  of the IEEE}, vol. 100, no.~1, pp. 195--209, Jan. 2012.

\bibitem{dan2010stealth}
G.~D\'{a}n and H.~Sandberg, ``Stealth attacks and protection schemes for state
  estimators in power systems,'' in \emph{First IEEE International Conference
  on Smart Grid Communications (SmartGridComm)}, 2010, pp. 214--219.

\bibitem{varshneyByzantine}
A.~S. Rawat, P.~Anand, H.~Chen, and P.~K. Varshney, ``Collaborative spectrum
  sensing in the presence of {B}yzantine attacks in cognitive radio networks,''
  \emph{IEEE Transactions on Signal Processing}, vol.~59, no.~2, pp. 774--786,
  Feb. 2011.

\bibitem{varshneyHypothesisGame}
W.~Hashlamoun, S.~Brahma, and P.~K. Varshney, ``Mitigation of {B}yzantine
  attacks on distributed detection systems using audit bits,'' \emph{IEEE
  Transactions on Signal and Information Processing over Networks}, vol.~4,
  no.~1, pp. 18--32, Mar. 2018.

\bibitem{detectionGameAdversary}
B.~Tondi, N.~Merhav, and M.~Barni, ``Detection games under fully active
  adversaries,'' \emph{Entropy}, vol.~21, no.~1, Jan. 2019, {A}rt. no. 23.

\bibitem{multiAntennaJamming}
V.~S.~S. Nadendla, V.~Sharma, and P.~K. Varshney, ``On strategic multi-antenna
  jamming in centralized detection networks,'' \emph{IEEE Signal Processing
  Letters}, vol.~24, no.~2, pp. 186--190, Feb. 2017.

\bibitem{gupta2010optimal}
A.~Gupta, C.~Langbort, and T.~Ba{\c{s}}ar, ``Optimal control in the presence of
  an intelligent jammer with limited actions,'' in \emph{49th IEEE Conference
  on Decision and Control (CDC)}, 2010, pp. 1096--1101.

\bibitem{gupta2012dynamic}
A.~Gupta, A.~Nayyar, C.~Langbort, and T.~Ba{\c{s}}ar, ``A dynamic
  transmitter-jammer game with asymmetric information,'' in \emph{51st IEEE
  Conference on Decision and Control (CDC)}, 2012, pp. 6477--6482.

\bibitem{basar1985complete}
T.~Ba\c{s}ar and Y.-W. Wu, ``A complete characterization of minimax and maximin
  encoder-decoder policies for communication channels with incomplete
  statistical description,'' \emph{IEEE Transactions on Information Theory},
  vol.~31, no.~4, pp. 482--489, 1985.

\bibitem{Poor}
H.~V. Poor, \emph{An Introduction to Signal Detection and Estimation},
  2nd~ed.\hskip 1em plus 0.5em minus 0.4em\relax New York, NY, USA:
  Springer-Verlag, 1994.

\bibitem{KayDetection}
S.~M. Kay, \emph{Fundamentals of Statistical Signal Processing, Volume II:
  Detection Theory}.\hskip 1em plus 0.5em minus 0.4em\relax Prentice-Hall,
  1993.

\bibitem{basols99}
T.~Ba\c{s}ar and G.~J. Olsder, \emph{Dynamic Noncooperative Game Theory}.\hskip
  1em plus 0.5em minus 0.4em\relax Philadelphia, PA: SIAM Classics in Applied
  Mathematics, 1999.

\bibitem{dynamicGameArxiv}
S.~Sar{\i}ta\c{s}, S.~Y\"uksel, and S.~Gezici, ``Dynamic signaling games with
  quadratic criteria under {N}ash and {S}tackelberg equilibria,''
  \emph{Automatica}, conditionally accepted, arXiv:1704.03816.

\bibitem{SignalingGames}
V.~P. Crawford and J.~Sobel, ``Strategic information transmission,''
  \emph{Econometrica}, vol.~50, pp. 1431--1451, 1982.

\bibitem{mismatchedEstimation}
C.~D. Richmond and L.~L. Horowitz, ``Parameter bounds on estimation accuracy
  under model misspecification,'' \emph{IEEE Transactions on Signal
  Processing}, vol.~63, no.~9, pp. 2263--2278, May 2015.

\bibitem{estimationRobustnessMismatch}
R.~M. Dufour and E.~L. Miller, ``Statistical estimation with 1/f-type prior
  models: robustness to mismatch and efficient model determination,'' in
  \emph{IEEE International Conference on Acoustics, Speech, and Signal
  Processing (ICASSP)}, vol.~5, May 1996, pp. 2491--2494.

\bibitem{mismatchSurvey}
S.~Fortunati, F.~Gini, M.~S. Greco, and C.~D. Richmond, ``Performance bounds
  for parameter estimation under misspecified models: Fundamental findings and
  applications,'' \emph{IEEE Signal Processing Magazine}, vol.~34, no.~6, pp.
  142--157, Nov. 2017.

\bibitem{BasTAC85}
T.~Ba\c{s}ar, ``An equilibrium theory for multiperson decision making with
  multiple probabilistic models,'' \emph{IEEE Transactions on Automatic
  Control}, vol.~30, no.~2, pp. 118--132, Feb. 1985.

\bibitem{TeneketzisVaraiya88}
D.~Teneketzis and P.~Varaiya, ``Consensus in distributed estimation,'' in
  \emph{Advances in Statistical Signal Processing}, H.~V. Poor, Ed.\hskip 1em
  plus 0.5em minus 0.4em\relax Greenwich: JAI Press, 1988, ch.~10, pp.
  361--386.

\bibitem{CastanonTeneketzis88}
D.~A. Castanon and D.~Teneketzis, ``Further results on the asymptotic agreement
  problem,'' \emph{IEEE Transactions on Automatic Control}, vol.~33, no.~6, pp.
  515--523, June 1988.

\bibitem{YukselBasarBook}
S.~Y\"uksel and T.~Ba\c{s}ar, \emph{Stochastic Networked Control Systems:
  Stabilization and Optimization under Information Constraints}.\hskip 1em plus
  0.5em minus 0.4em\relax Boston, MA: Birkh\"auser, 2013.

\bibitem{tacWorkArxiv}
S.~Sar{\i}ta\c{s}, S.~Y{\"{u}}ksel, and S.~Gezici, ``Quadratic
  multi-dimensional signaling games and affine equilibria,'' \emph{IEEE
  Transactions on Automatic Control}, vol.~62, no.~2, pp. 605--619, Feb. 2017.

\bibitem{omerHierarchial}
M.~O. Sayin, E.~Akyol, and T.~Ba\c{s}ar, ``Hierarchical multistage {G}aussian
  signaling games in noncooperative communication and control systems,''
  \emph{Automatica}, vol. 107, pp. 9--20, 2019.

\bibitem{securityStackelberg}
G.~Brown, M.~Carlyle, J.~Salmerón, and K.~Wood, ``Defending critical
  infrastructure,'' \emph{Interfaces}, vol.~36, no.~6, pp. 530--544, 2006.

\bibitem{NashvsStackelberg}
D.~Korzhyk, Z.~Yin, C.~Kiekintveld, V.~Conitzer, and M.~Tambe, ``Stackelberg
  vs. {N}ash in security games: An extended investigation of
  interchangeability, equivalence, and uniqueness,'' \emph{Journal of
  Artificial Intelligence Research}, vol.~41, no.~2, pp. 297--327, May 2011.

\bibitem{CedricWork}
F.~Farokhi, A.~M.~H. Teixeira, and C.~Langbort, ``Estimation with strategic
  sensors,'' \emph{IEEE Transactions on Automatic Control}, vol.~62, no.~2, pp.
  724--739, Feb. 2017.

\bibitem{akyolITapproachGame}
E.~Akyol, C.~Langbort, and T.~Ba\c{s}ar, ``Information-theoretic approach to
  strategic communication as a hierarchical game,'' \emph{Proceedings of the
  IEEE}, vol. 105, no.~2, pp. 205--218, Feb. 2017.

\bibitem{sourceIdentification}
M.~Barni and B.~Tondi, ``The source identification game: An
  information-theoretic perspective,'' \emph{IEEE Transactions on Information
  Forensics and Security}, vol.~8, no.~3, pp. 450--463, Mar. 2013.

\bibitem{inspectionGames}
R.~Avenhaus, B.~von Stengel, and S.~Zamir, ``Inspection games,'' in
  \emph{Handbook of Game Theory with Economic Applications}, 1st~ed., R.~J.
  Aumann and S.~Hart, Eds.\hskip 1em plus 0.5em minus 0.4em\relax Elsevier,
  2002, vol.~3, ch.~51, pp. 1947--1987.

\bibitem{Avenhaus1994}
R.~Avenhaus, ``Decision theoretic analysis of pollutant emission monitoring
  procedures,'' \emph{Annals of Operations Research}, vol.~54, no.~1, pp.
  23--38, Dec. 1994.

\bibitem{avenhausInspectLeadership}
------, ``Monitoring the emission of pollutants by means of the inspector
  leadership method,'' in \emph{Conflicts and Cooperation in Managing
  Environmental Resources}, R.~Pethig, Ed.\hskip 1em plus 0.5em minus
  0.4em\relax Springer Berlin Heidelberg, 1992, pp. 241--273.

\bibitem{htDeception}
T.~Zhang and Q.~Zhu, ``Hypothesis testing game for cyber deception,'' in
  \emph{Decision and Game Theory for Security}, L.~Bushnell, R.~Poovendran, and
  T.~Ba{\c{s}}ar, Eds.\hskip 1em plus 0.5em minus 0.4em\relax Cham: Springer
  International Publishing, 2018, pp. 540--555.

\bibitem{MuratA}
M.~Azizoglu, ``Convexity properties in binary detection problems,'' \emph{IEEE
  Transactions on Information Theory}, vol.~42, no.~4, pp. 1316--1321, July
  1996.

\end{thebibliography}

\end{document}